\documentclass[11pt,a4paper]{amsart}[2000]
\usepackage{amsmath}
\usepackage{amssymb}
\usepackage{amsthm}
\usepackage[hang]{footmisc}
\usepackage[all]{xypic}
\usepackage{hyperref}

\title{The nuclear dimension of $C^*$-algebras associated to homeomorphisms}

\author{Ilan Hirshberg}

\address{Department of Mathematics, Ben Gurion University of the Negev, \phantom{----------------}\linebreak\text{}\hspace{3.5mm}
P.O.B. 653, Be'er Sheva 84105, Israel}
\email{ilan@math.bgu.ac.il}

\author{Jianchao Wu}

\address{Westf{\"a}lische Wilhelms-Universit{\"a}t, Fachbereich Mathematik, \phantom{---------------------------}\linebreak \text{}\hspace{3.5mm} Einsteinstrasse 62, 48149 M{\"u}nster, Germany}
\email{jianchao.wu@uni-muenster.de}

\thanks{This research was supported by GIF grant 1137/2011, ERC Advanced Grant ToDyRiC 267079, and SFB 878 \emph{Groups, Geometry and Actions}.}

\allowdisplaybreaks

\theoremstyle{plain}

\newtheorem{Thm}{Theorem}[section]

\newtheorem{Lemma}[Thm]{Lemma}

\newtheorem{Prop}[Thm]{Proposition}

\theoremstyle{definition}

\newtheorem{Rmk}[Thm]{Remark}

\theoremstyle{plain}
\newtheorem{thm}[Thm]{Theorem}
\newtheorem{lem}[Thm]{Lemma}
\newtheorem{cor}[Thm]{Corollary}
\newtheorem{prop}[Thm]{Proposition}

\newtheorem{clm}[Thm]{Claim}

\theoremstyle{definition}
\newtheorem{defn}[Thm]{Definition}

\newtheorem{rmk}[Thm]{Remark}

\newcommand{\T}{{\mathbb T}}

\newcommand{\N}{{\mathbb N}}
\newcommand{\Z}{{\mathbb Z}}
\newcommand{\C}{{\mathbb C}}

\newcommand{\aut}{\mathrm{Aut}}

\newcommand{\eps}{\varepsilon}
\numberwithin{equation}{section}

\newcommand{\dr}{\mathrm{dr}}
\newcommand{\id}{\mathrm{id}}

\newcommand{\dimnuc}{\mathrm{dim}_{\mathrm{nuc}}}
\newcommand{\dimnucone}[0]{\dimnuc^{\!+1}}

\newcommand{\calpha}{\widehat{\alpha}}

\newcommand\set[1]{\left\{#1\right\}}  

 \newcommand{\IN}[0]{\mathbb{N}}

\newcommand{\CA}[0]{\mathcal{A}} \newcommand{\CB}[0]{\mathcal{B}}

\newcommand{\quer}[0]{\overline}

\newcommand{\fin}[0]{{\subset\!\!\!\subset}}

\newcommand{\del}[0]{\partial}

\newtheorem{lemma}[Thm]{Lemma}

\theoremstyle{definition}
\newtheorem{defi}[Thm]{Definition}

\numberwithin{equation}{Thm}

\begin{document}
\begin{abstract}
We show that if $X$ is a finite dimensional locally compact Hausdorff space, then the crossed product of $C_0(X)$ by any automorphism has finite nuclear dimension. This generalizes previous results, in which the automorphism was required to be free. As an application, we show that group $C^*$-algebras of certain non-nilpotent groups have finite nuclear dimension.
\end{abstract}
\maketitle

Nuclear dimension for $C^*$-algebras was introduced by Winter and Zacharias in \cite{winter-zacharias}, as a noncommutative generalization of covering dimension. This is a variant of the previous notion of decomposition rank (\cite{kirchberg-winter}), and is also applicable to non-quasidiagonal $C^*$-algebras. Since then, it has come to play a major role in structure and classification of $C^*$-algebras. It was shown in \cite{winter-zacharias} that if $X$ is a locally compact metrizable space, then $\dimnuc(C_0(X))$ coincides with the covering dimension of $X$, and the property of having finite nuclear dimension is preserved under various constructions: forming direct sums and tensor products, passing to quotients and hereditary subalgebras, and forming extensions. An important problem which was left open in \cite{winter-zacharias} is to understand the behavior of finite nuclear dimension under forming crossed products. It was shown in \cite{toms-winter} that finite nuclear dimension passes to crossed products by minimal homeomorphisms:  if $X$ is a compact metric space with finite covering dimension and $h \colon X \to X$ is a minimal homeomorphism, then denoting $\alpha(f) = f \circ h$, we have $\dimnuc(C(X) \rtimes_{\alpha}\Z) < \infty$. This was re-proved in a different way in \cite{HWZ}. The paper \cite{HWZ} develops a notion of \emph{Rokhlin dimension} for an automorphism of a $C^*$-algebra (extended in \cite{hirshberg-phillips} to the non-unital setting). It was shown there that in general, if $A$ has finite nuclear dimension and $\alpha \in \aut(A)$ has finite Rokhlin dimension, then $A \rtimes_{\alpha} \Z$ has finite nuclear dimension as well, and furthermore, for a minimal homeomorphism as above, the induced automorphism on $C(X)$ always has finite Rokhlin dimension. Szab\'o (\cite{szabo}) then showed that the minimality condition can be weakened to freeness: if $X$ is as above and $h \colon X \to X$ has no periodic points, then $\alpha$ has finite Rokhlin dimension (and therefore, by \cite[Theorem 4.1]{HWZ}, the crossed product has finite nuclear dimension). In fact, Szab\'o's result works for actions of $\Z^m$ as well. This uses the marker property, introduced by Gutman in \cite{gutman}. Those results were further extended to free actions of finitely generated nilpotent groups in \cite{SWZ}.

For the case of integer actions arising from homeomorphisms, this leaves the case of actions which also have periodic points. Those include important examples. For instance, suppose $G$ is a countable abelian group and $\widehat{G}$ has finite covering dimension, and suppose $\alpha$ is an automorphism of $G$. The group $C^*$-algebra $C^*(G \rtimes_{\alpha} \Z)$ is isomorphic to a crossed product $C(\widehat{G}) \rtimes \Z$, and such actions are never free: an obvious fixed point is $1_{\widehat{G}}$, and in some cases, such as the lamplighter group, the set of periodic points may even be dense.

In this paper, we settle the case of crossed products arising from homeomorphisms. We show in Theorem \ref{thm:estimate-dimnuc-Z} that if $X$ is a locally compact metrizable space with finite covering dimension and $\alpha \in \aut(C_0(X))$ then
\[
  \dimnuc(C_0(X) \rtimes_{\alpha} \Z) \leq 2 \big( \dim(X) \big)^2 + 6 \dim(X) +4 \; .
\]
The non-metrizable case is addressed in Corollary \ref{cor:nonmetrizable}. 

As indicated above, our formula implies that the group $C^*$-algebra $C^*(G \rtimes_{\alpha} \Z)$ has finite nuclear dimension, whenever $G$ is abelian with finite dimensional Pontryagin dual. Notable examples of such groups include:
 \begin{enumerate}
  \item The lamplighter group $(\Z / 2 \Z) \wr \Z$ is the semidirect product of $G = \bigoplus_{n \in \Z} \Z / 2\Z$ by $\Z$ using the shift action. Notice that by \cite[Corollary 3.5]{CDE13}, the group $C^*$-algebra $C^*((\Z / 2 \Z) \wr \Z)$ is not strongly quasidiagonal, and thus it has infinite decomposition rank (\cite[Theorem 4.4]{kirchberg-winter}). This shows that there exists a group $C^*$-algebra which has finite nuclear dimension but infinite decomposition rank.
  \item When $G = \Z^n$, we obtain polycyclic groups which may not be nilpotent. More precisely, if we pick a matrix $A \in \mathrm{GL}_n(\Z)$ such that $A$ has no eigenvalues which are roots of unity and use it to define an automorphism of $\Z^n$, then the crossed product $\Z^n \rtimes_A \Z$ by this automorphism has trivial center, and in particular it is not nilpotent.
 \end{enumerate}
 
We note that nuclear dimension of group $C^*$-algebras was studied recently in \cite{eckhardt-mckenney}. It was shown there that if $G$ is a finitely generated nilpotent group then $C^*(G)$ has finite nuclear dimension. The following question now appears natural:\\

\paragraph{\textbf{Question:}} Let $G$ be a (virtually) polycyclic group. Does $C^*(G)$ have finite nuclear dimension? What about elementary amenable groups with finite Hirsch lengths?\\

In the remainder of the introduction, we sketch the idea of our proof. We note that we cannot directly use the previously known results concerning Rokhlin dimension, since actions with finite Rokhlin dimension are necessarily free. Let us first suppose we are in the other extreme: the homeomorphism $h$ under consideration is periodic, with $h^n = \id$. In this case, although the action does not have finite Rokhlin dimension, one can show directly that the crossed product has finite nuclear dimension (in fact, finite decomposition rank): the crossed product is subhomogeneous, so we only need to find a bound on the dimensions of the spaces of irreducible representations of different dimensions and appeal to \cite{winter-subhomogeneous}. (In this setting, however, one has more information about the structure of the crossed product; this allows us to provide a short and more direct proof, which will also be applicable for actions of groups other than $\Z$). A key fact here is that the bound on $\dimnuc(C_0(X) \rtimes \Z)$ does not depend on the period $n$ of the action, but only on $\dim(X)$. 

Next, let us consider the somewhat more complicated case, in which there are both periodic and non-periodic points, but there is a bound on the length of the orbits: each point is either periodic with period at most $n$, or acted on freely by $h$. In such a case, if we denote by $X_{\mathrm{periodic}}$ the set of all periodic points and by $X_{\mathrm{free}}$ the set of all points on which $h$ acts freely, then $X_{\mathrm{periodic}}$ is a closed invariant set and we have an equivariant extension 
$$
0 \to C_0(X_{\mathrm{free}}) \to C_0(X) \to C_0(X_{\mathrm{periodic}}) \to 0 
.
$$
An extension of Szab\'o's arguments (\cite{szabo}) to the non-compact setting shows that the restriction of $\alpha$ to the ideal $C_0(X_{\mathrm{free}})$ has finite Rokhlin dimension. One can now use the fact that finite nuclear dimension is preserved by extensions. 

Of course, in general there may be no bound on the length of the orbits, and the set of periodic points need not be closed. However, if we fix some $N$, we can consider  the set of points which are periodic with orbit length $\leq N$, which we denote by $X_{\leq N}$, and we let $X_{>N} = X \smallsetminus X_{\leq N}$. Then $X_{\leq N}$ is a closed subset, and again we have an equivariant extension 
$$
0 \to C_0(X_{> N}) \to C_0(X) \to C_0(X_{\leq N}) \to 0
.
$$
As discussed above, we have a bound on $\dimnuc(C_0(X_{\leq N}) \rtimes \Z)$ which does not depend on $N$. As for $X_{> N}$, although in general the restriction of $\alpha$ to $C_0(X_{>N})$ does not have finite Rokhlin dimension, we can still use a refined version of the marker property, detailed in the appendix, to show that $\alpha$ satisfies some fragment of the definition of finite Rokhlin dimension. Recall that to have finite Rokhlin dimension, the automorphism $\alpha$ should admit arbitrarily long Rokhlin towers, each consisting of positive contractions permuted by $\alpha$ to within any given tolerance. Here, we can find Rokhlin towers, provided they are not too long and the error is not too small (compared to $N$). This will be made precise in Lemma \ref{lem:tower-Z}. Finally, in order to construct a decomposable approximation for a given finite subset of $C_0(X) \rtimes_{\alpha} \Z$ to within a specified tolerance, we can choose $N$ to be large enough so as to be able to construct sufficiently long Rokhlin towers with a sufficiently small error, and then apply a localized version of the argument which shows that finite nuclear dimension passes to extensions.

The paper is organized as follows. We begin by fixing notation and listing a few general lemmas. In Section \ref{section:reduction to SH}, we find an upper bound on the nuclear dimension of the crossed product of $C_0(X)$ by a periodic action, using Winter's bound for subhomogeneous algebras. In Section \ref{section: uniformly compact orbits}, we obtain an upper bound via a different approach. Although the second upper bound we obtain is higher than the one we find in Section \ref{section:reduction to SH}, the method works for groups other than $\Z$, which we hope will be useful for future work.  In Section \ref{section: long orbits} we show that there exist Rokhlin towers (of certain length and tolerance) for homeomorphisms whose orbits are all sufficiently long. In the last section, we combine those results to derive our main theorem, Theorem \ref{thm:estimate-dimnuc-Z}. The appendix, by Szab\'o, contains the refinement of the marker property needed in Section \ref{section: long orbits}.

The authors are grateful to G\'abor Szab\'o for providing the said refinement of the marker property method and presenting it in the appendix. The second author would also like to thank Caleb Eckhardt, Stuart White and Joachim Zacharias for pointing out certain applications of our results.

\section{Preliminaries}
\label{section:prelim}

Throughout the paper, we use the following conventions. To simplify formulas, we use the notations $\dimnucone(A) = \dimnuc(A)+1$, $\dim^{+1}(X) = \dim(X)+1$ and $\dr^{+1}(A) = \dr(A)+1$. If $A$ is a $C^*$-algebra, we denote by $A_+$ the positive part, and by $A_{+, \leq 1}$ the set of positive elements of norm at most $1$. If $G$ is a locally compact Hausdorff group and $A$ is a $C^*$-algebra, we denote by $\alpha \colon G \curvearrowright A$ an action, that is, a continuous homomorphism $\alpha \colon G \to \aut(A)$, where $\aut(A)$ is topologized by pointwise convergence. If $G = \Z$, we shall often denote by $\alpha$ both the action and the homomorphism $\alpha_1$ which generates it, when it causes no confusion. 

We are interested here in the case in which $A$ is commutative, that is, $A \cong C_0(X)$ for some locally compact Hausdorff space $X$ (namely the spectrum $\widehat{A}$). By Gel'fand's theorem, an action $\alpha \colon G \curvearrowright C_0(X)$ is completely determined by a continuous action $\calpha \colon G  \curvearrowright X$ on the spectrum, and vice versa. They are related by the identity $\alpha_g (f) = f \circ \calpha_{g^{-1}}$ for any $f \in C_0(X)$ and any $g \in G$. Thus taking a $C^*$-algebraic point of view, we will denote by $\calpha \colon G  \curvearrowright X$ an action on a locally compact Hausdorff space by homeomorphisms, and save the notation $\alpha$ for the corresponding action on $C_0(X)$. In the case of $G = \Z$, we shall also use $\calpha$ to denote the homeomorphism given by the generator $1$. 

If $A = A_0 \oplus A_1 \oplus \ldots \oplus A_d$ is a $C^*$-algebra, and $\varphi^{(k)} \colon A_k \to B$ are order zero contractions into some $C^*$-algebra $B$ for $k=0,1,\ldots,d$, we say that the map $\varphi = \sum_{k=0}^d \varphi^{(k)}$ is a piecewise contractive $(d+1)$-decomposable completely positive map.

The following fact concerning order zero maps is standard and used often in the literature. It follows immediately from the fact that cones over finite dimensional $C^*$-algebras are projective. See \cite[Proposition 1.2.4]{winter-covering-II} and the proof of \cite[Proposition 2.9]{winter-zacharias}. We record it here for further reference.

\begin{Lemma}
\label{Lemma:lifting-decoposable-maps}
Let $A$ be a finite dimensional $C^*$-algebra, let $B$ be a $C^*$-algebra and let $I \lhd B$ be an ideal. Then any piecewise contractive $(d+1)$-decomposable completely positive map $\varphi \colon A \to B/I$ lifts to a piecewise contractive $(d+1)$-decomposable completely positive map $\widetilde{\varphi} \colon A \to B$. \qed
\end{Lemma}

The following technical lemma is straightforward, and variants of it have been used in the literature. We include a short proof for the reader's convenience.
\begin{Lemma}
 \label{Lemma:finite-dimnuc}
 Let $B$ be a separable and nuclear $C^*$-algebra and $B_0$ a dense subset of the unit ball of $B$. Then $\dimnuc(B) \leq d$ if and only if for any finite subset $F \subseteq B_0$ and for any $\eps>0$ there exists a $C^*$-algebra $A_{\eps} = A_{\eps}^{(0)} \oplus \cdots \oplus A_{\eps}^{(m)}$ and completely positive maps 
 \[
  \xymatrix{
   B \ar[dr]_{\psi = \bigoplus_{l=0}^m \psi^{(l)} \quad } \ar@{.>}[rr]^{\id} &  & B \\
   & A_{\eps} = \bigoplus_{l=0}^m A_{\eps}^{(l)} \ar[ur]_{\quad\varphi = \sum_{l=0}^m \varphi^{(l)}} &
  }
 \]
 so that
 \begin{enumerate}
  \item $\psi$ is contractive,
  \item each $\varphi^{(l)}$ is a sum $\varphi^{(l)} = \sum_{k=0}^{ d^{(l)} } \varphi^{(l,k)}$ of $(d^{(l)} + 1)$-many order zero contractions,
  \item $\|\varphi(\psi(x)) - x\| < \eps$ for all $x \in F$, and
  \item $\displaystyle \sum_{l=0}^m ( \dimnuc(A_{\eps}^{(l)}) + 1) (d^{(l)}+1) \leq d+1$.
 \end{enumerate}
\end{Lemma}
\begin{proof}
The forward implication is immediate from the definition of nuclear dimension. For the converse, let $F \subset B$ be a finite set, and fix $\eps>0$. We wish to find a piecewise contractive $(d+1)$-decomposable completely positive approximation for $F$ through a finite dimensional $C^*$-algebra. Since $B_0$ is dense in the unit ball of $B$, we may assume that $F \subset B_0$, by applying a rescaling and a small perturbation if needed. Let $A_{\eps}$, $\psi$ and $\varphi$ be as in the statement. Set $\eps' = \max\{\|\varphi(\psi(x)) - x\|
\mid x \in F\}$, and note that $\eps'<\eps$. For each $l = 0,1,\ldots,m$, pick a piecewise contractive $\dimnucone(A_{\eps}^{(l)})$-decomposable approximation for $\psi^{(l)}(F)$ to within $(\eps-\eps')/(m+1)$, 
 \[
  \xymatrix{
   A_{\eps}^{(l)} \ar[dr]_{\sigma^{(l)} \quad } \ar@{.>}[rr]^{\id} &  & A_{\eps}^{(l)} \\
   & E^{(l)} \ar[ur]_{\quad\eta^{(l)}}  &
   } 
 \]
 One now checks that
  \[
   \xymatrix{
    B \ar[dr]_{ \bigoplus_{l=0}^m \sigma^{(l)} \circ \psi^{(l)} \quad } \ar@{.>}[rr]^{\id} &  & B \\
    & \bigoplus_{l=0}^m E^{(l)} \ar[ur]_{\quad \sum_{l=0}^m \varphi^{(l)} \circ \eta^{(l)}}  &
   }
  \]
  is a decomposable approximation, as required.
\end{proof}

 The following lemma is an invariant version of \cite[Proposition 2.6]{winter-zacharias}. The modification is straightforward as well, but we include a proof for the reader's convenience.

\begin{lem}\label{lem:separable-dimnuc}
 Let $G$ be a locally compact Hausdorff and second countable group, and let $A$ be a $G$-$C^*$-algebra. Then any countable subset $S \subset A$ is contained in a $G$-invariant separable $C^*$-subalgebra $B \subset A$ with $\dimnuc(B) \leq \dimnuc(A)$. In particular, $A$ can be written as a direct limit of separable $G$-$C^*$-algebras with nuclear dimension no more than $\dimnuc(A)$.
\end{lem}
\begin{proof}
We define an increasing sequence of separable $G$-invariant $C^*$-subalgebras of $A$ as follows. Let $B_0$ be the $G$-$C^*$-subalgebra of $A$ generated by $S$. Now, suppose $B_n$ has been defined. We pick a countable dense sequence $x_1,x_2,\ldots$ in $B_n$. For any $k$, pick a piecewise contractive $\dimnucone(A)$-decomposable approximation 
 \[
  \xymatrix{
   A \ar[dr]_{\sigma_k \quad } \ar@{.>}[rr]^{\id} &  & A \\
   & E_k \ar[ur]_{\quad\eta_k = \sum_{j=0}^{\dimnuc(A)} \eta_k^{(j)}} &
   }  
 \]
 for $\{x_1,x_2,\ldots,x_k\}$ to within tolerance $\frac{1}{k}$. We set $B_{n+1} \subset A$ to be the $G$-$C^*$-subalgebra generated by $B_n$ and the images of $\eta_k^{(j)}$ for all applicable $k$ and $j$. Since each $E_k$ is finite dimensional, and $G$ is second countable, the algebra $B_{n+1}$ is separable. Furthermore, by construction, for any finite subset $F \subset B_n$ and any $\eps>0$, we may choose $k$ large enough so that the diagram
 \[
  \xymatrix{
   A \ar[dr]_{\sigma_k \quad } \ar@{.>}[rr]^{\id} &  & A \\
   & E_k \ar[ur]_{\quad\eta_k = \sum_{j=0}^{\dimnuc(A)} \eta_k^{(j)}}  &
  } 
 \]
 gives a piecewise contractive $\dimnucone(A)$-decomposable approximation for $(F, \eps)$ whose image lies in $B_{n+1}$. We now define $B = \overline{\bigcup_{n=0}^{\infty}B_n} \subset A$. This is a separable $G$-$C^*$-subalgebra. We claim that $\dimnuc(B) \leq \dimnuc(A)$. Indeed, for any finite set $F \subset \bigcup_{n=0}^{\infty}B_n$ (which is dense in $B$) and any $\eps>0$, there is a  piecewise contractive $\dimnucone(A)$-decomposable approximation for $(F, \eps)$ whose image lies in $B$. Restricting the domain and co-domain to $B$ gives us the required approximation.
 \end{proof}

\begin{lem}\label{lem:quasicentral-approximate-unit}
 Let $X$ be a locally compact Hausdorff space, let $G$ be a locally compact Hausdorff group, and let $\calpha: G \curvearrowright X$ be a continuous action. Suppose $U$ is a $G$-invariant open subset of $X$. Then there is a quasicentral approximate unit for $C_0(U) + C_0(U) \rtimes_\alpha G \subset M(C_0(X) \rtimes_\alpha G)$ which is contained in $C_c(U)_{+, \leq 1}$.
\end{lem}

\begin{proof}
 Let $K$ be a compact subset of $U$, and let $S$ be a symmetric compact neighborhood of the identity in $G$. Then $S\cdot K$ is also a compact subset of $U$. If $e \in C_c(U)$ is a positive contraction which is identically $1$ on $S\cdot K$, then for any function $f \in C_c(G,C_c(U)) \subseteq A \rtimes_{\alpha} G$ such that $f$ is supported in $S$ and $f(g)$ is supported in $K$ for each $g \in S$, we have $fe = ef = f$. Therefore, for any finite subset $F$ of $C_c(G,C_c(U))$ there exists a positive contraction $e_F \in C_c(U)$ so that $e_F$ acts as the identity on $F$.  Since $C_c(G,C_c(U))$ is dense in $A \rtimes_{\alpha} G$, it follows that there exists an approximate identity for $C_0(U) + C_0(U) \rtimes_\alpha G$ whose elements are all in $C_c(U)_{+, \leq 1}$. By the remark after \cite[Theorem 1]{arveson}, it follows that there exists a quasicentral approximate unit for $C_0(U) + C_0(U) \rtimes_\alpha G \subset M(C_0(X) \rtimes_\alpha G)$ in the convex hull of the elements in $C_c(U)$ described above, and in particular it is contained in $C_c(U)_{+, \leq 1}$, as required.
\end{proof}

We record the following two results from classical dimension theory, which are used later in the paper.
Those two results apply to the case of metrizable spaces, since any metrizable space is paracompact, Hausdorff and totally normal. For a discussion of different variants of paracompactness and normality, we refer the reader to \cite[Chapter 1, section 4]{Pears75}.

\begin{thm}[{\cite[Chapter 3, Theorem 6.4]{Pears75}}]\label{thm:Pears75}
 If $M$ is a subspace of a totally normal space $X$, then $\dim (M) \leq \dim (X)$.
\end{thm}

\begin{prop}[{\cite[Chapter 9, Proposition 2.16]{Pears75}}] \label{prop:Pears75}
 If $X$ and $Y$ are weakly paracompact normal Hausdorff spaces and $f\colon X \to Y$ is a continuous open surjection such that $f^{-1}(y)$ is finite for each point of $Y$, then $\dim(X) = \dim (Y) $. \qed
\end{prop}

\section{Crossed products by a periodic automorphism} 
\label{section:reduction to SH}

We identify the crossed product of $C_0(X)$ by a periodic action as a subhomogeneous algebra, and use Winter's method to provide an upper bound on its decomposition rank (and thus on its nuclear dimension, too).

\begin{Prop} 
\label{Prop:periodic points}
Suppose $Y$ is a locally compact metrizable space of finite covering dimension. Let $\alpha \colon C_0(Y) \to C_0(Y)$ be a periodic automorphism, that is, $\alpha^n = \id$ for some positive integer $n$. Then 
  $C_0(Y) \rtimes_{\alpha} \Z$ is subhomogeneous
 and 
 $$
 \dr(C_0(Y) \rtimes_{\alpha} \Z) \leq \dim(Y)+1 
 .
 $$
\end{Prop}

\begin{proof}
That $C_0(Y) \rtimes_{\alpha} \Z$ is subhomogeneous follows from the more general fact that if $\alpha$ is a periodic automorphism of a $C^*$-algebra $A$ with $\alpha^n = \id$, then $A$ embeds in $M_n(A) \otimes C(\T)$. We give the full details here, since we need a concrete description of the primitive ideal space of the crossed product to establish the bound on the decomposition rank of the crossed product.

We embed $C_0(Y) \rtimes_{\alpha} \Z$ in $M_n(C_0(Y) \otimes C(\T)) \cong M_n \otimes C_0(Y) \otimes C(\T) $ as follows. Let $z \in C(\T)$ be the standard generator. (Think of $\T$ as the unit circle in $\C$, and $z$ as the inclusion map.) Define $\beta \colon C_0(Y) \to M_n(C_0(Y)) \otimes C(\T)$ by $\beta(f) = \mathrm{diag}(f,\alpha(f),\ldots,\alpha^{n-1}(f)) \otimes 1_{C(\T)}$, and $u \in M (M_n(C_0(Y)) \otimes C(\T)) \cong M (M_n(C_0(Y) \otimes C(\T)))$ by 
$$
u = \left ( 
\begin{matrix} 0 & 1 & 0 &  \cdots & 0  \\
				0 & 0 & 1 &  \cdots & 0 \\
			   \vdots & \vdots  & \vdots & \ddots &\vdots   \\
			   0 & 0 & 0 &  \cdots & 1 \\
				1 \otimes z & 0 & 0 & \cdots & 0 
\end{matrix}
\right )
\, .
$$
Then $(\beta,u)$ is a covariant representation of $(C_0(Y),\alpha)$ (that is, $u\beta(f)u^* = \beta(\alpha(f))$). The associated homomorphism $\widetilde{\beta} \colon C_0(Y) \rtimes_{\alpha} \Z \to M_n(C_0(Y)) \otimes C(\T)$ is injective. To see that, let $B = C^*(\{\beta(f)u^n \mid f \in C_0(Y), n \in \Z\})$ be the image of $\widetilde{\beta}$. For any $\lambda \in \T$, set $v_{\lambda} = \mathrm{diag}(\bar{\lambda},\bar{\lambda}^2,\ldots,\bar{\lambda}^n)$. Let $\eta_{\lambda} \in \aut(C(\T))$ be the rotation automorphism given by $\eta_{\lambda}(z) = \lambda^nz$. Let $\gamma_{\lambda} \in \aut(M_n\otimes C_0(Y)  \otimes C(\T))$ be the automorphism given by $\gamma_{\lambda} = \mathrm{Ad}(v_{\lambda}) \circ (\id \otimes \id \otimes \eta_{\lambda})$. Then $B$ is invariant under $\gamma_{\lambda}$ for any $\lambda \in \T$. We restrict $\gamma_{\lambda}$ to $B$ and keep the same notation. Then $\gamma_{\lambda}(u) = \lambda u$ (where we extend $\gamma_{\lambda}$ to the multiplier algebra if need be), and $\gamma_{\lambda}(\beta(f)) = \beta(f)$ for all $f \in C_0(Y)$. Let $E \colon C_0(Y) \rtimes_{\alpha} \Z \to C_0(Y)$ be the canonical expectation, and let $E_{\gamma} \colon B \to \beta(C_0(Y))$ be the expectation given by integrating along $\lambda$. We have a commuting diagram:
$$
\xymatrix{
C_0(Y)\rtimes_{\alpha} \Z  \ar[r]^{\widetilde{\beta}} \ar[d]_{E} & B \ar[d]^{E_{\gamma}} \\
C_0(Y) \ar[r]_{\beta} & \beta(C_0(Y))
}
$$
Since $\beta$ is injective and the two expectations are faithful, the map $\widetilde{\beta}$ is an isomorphism. It follows that $C_0(Y) \rtimes_{\alpha}\Z$ is subhomogeneous. 

By what we have just shown, we can identify $B$ with $C_0(Y) \rtimes_{\alpha}\Z$. Since $B \subseteq M_n(C_0(Y)\otimes C(\T))$, all irreducible representations of $B$ have dimension at most $n$. Fix $k \leq n$. We denote by $\mathrm{Prim}_k(B)$ the primitive ideal space associated to irreducible representations of $B$ on $M_k$. We denote by $Y_k$ the set of all points whose orbit consists of exactly $k$ points. We claim that $\mathrm{Prim}_k(B) \cong (Y_k / \Z) \times \T$. 

Let $\varphi \colon C_0(Y) \rtimes_{\alpha} \Z \to M_k$ be an irreducible representation. The restriction $\varphi|_{C_0(Y)}$ is unitarily equivalent to a representation of the form 
\begin{equation}\label{eq:standard-form-f}
 f \mapsto \mathrm{diag}(f(y_1),f(y_2),\ldots,f(y_k))
\end{equation}
for some $y_1,y_2,\ldots,y_k \in Y$. Set $v = \varphi(u)$. We claim that those points are distinct, and constitute a $k$-periodic orbit of the action. Fix $j \in \{1,2,\ldots,k\}$. Suppose $f \in C_0(Y)$ is a positive element such that $f(y_j) = 1$, $f(y)<1$ for all $y \in Y \smallsetminus \{y_j\}$, and $f(y) = 0$ for all other points in the orbit of $y_j$ and for all $y \in \{y_1,y_2,\ldots,y_k\} \smallsetminus \{y_j\}$.  
Let $k'$  be the period of $y_j$. Then $\varphi(\alpha^l(f))$ are mutually orthogonal projections for $l=0,1,\ldots,k'-1$. Thus, the orbit of $y_j$ is contained in $\{y_1,y_2,\ldots,y_k\}$. To see that the orbit of $y_j$ in fact equals $\{y_1,y_2,\ldots,y_k\}$, note that $\sum_{l=0}^{k'-1} \varphi(\alpha^l(f))$ is a projection which commutes with $v$, and since $\varphi$ is irreducible, it equals $1$, so evaluation at the points of the orbit of $y_j$ coincides with evaluation at $\{y_1,y_2,\ldots,y_k\}$. Likewise, if $y_j$ is repeated $m$ times in the sequence $y_1,y_2,\ldots,y_k$, then any other element of the orbit of $y_j$ is repeated $m$ times, since
$\varphi(f)$ is unitarily equivalent to $\varphi(\alpha^l(f))$ for all $l$. If $m>1$, we can pick $p = \mathrm{diag}(1,0,0,\ldots,0)$ 
using this diagonalization, and then $\sum_{l=0}^{k'-1}v^lpv^{*l}$ is a nontrivial projection in $\varphi(B)'$, which cannot happen. Therefore, $k=k'$ and $\{y_1,y_2,\ldots,y_k\}$ is an orbit.

Through a unitary equivalence, we may assume that $\calpha^{-1}(y_j) = y_{j+1}$ for $j \in \{1, \ldots, k-1 \}$ and $\calpha^{-1}(y_k) = y_{1}$. Thus $v$ is forced to be of the form 
\begin{equation}\label{eq:standard-form-u}
 \left ( 
 \begin{matrix} 0 & \lambda_1 & 0 &  \cdots & 0  \\
				0 & 0 & \lambda_2 &  \cdots & 0 \\
			   \vdots & \vdots  & \vdots & \ddots &\vdots   \\
			   0 & 0 & 0 &  \cdots & \lambda_{k-1} \\
				\lambda_k & 0 & 0 & \cdots & 0 
 \end{matrix}
 \right )
\end{equation}
with $\lambda_1, \ldots, \lambda_k \in \T$. 

Now since $v^k = \varphi(u^k) \in \varphi(B)'$ and $\phi$ was assumed to be irreducible, there exists a $\lambda \in \T$ such that $v^k = \lambda 1_k$. The choice of orbit and this $\lambda$ define a map $\Psi \colon \mathrm{Prim}_k(B) \to (Y_k/\Z) \times \T$. We claim that $\Psi$ is a homeomorphism.

We first check that $\Psi$ is continuous. If $\varphi \colon B \to M_k$ is an irreducible representation, then a direct computation shows $\lambda = \prod_{j=1}^k \lambda_j = (-1)^{k+1}\det(\varphi(u))$, whence the second component is continuous. As for the first component, pick $[y] \in Y_k /\Z$ and a neighborhood $U$. Let $\{y_1,y_2,\ldots,y_k\}$ be the orbit of $y$, and let $V \subseteq Y$ be an open neighborhood of $\{y_1,y_2,\ldots,y_k\}$ such that $(V \cap Y_k) /\Z \subset U$. Pick $f \in C_0(V)$ which is $1$ on $\{y_1,y_2,\ldots,y_k\}$. Thus, the preimage of $U$ under the first component map of $\Psi$ contains the open set $\{[\pi] \mid \pi(f) \neq 0\}$, which is a neighborhood of the set of irreducible representations associated to this orbit. Therefore, $\Psi$ is continuous.

Given $([y],\lambda) \in (Y_k/\Z) \times \T$, we can consider the covariant representation given by 
\[
\varphi_y(f) = \left ( 
\begin{matrix} f(y) & 0 & 0 &  \cdots & 0 & 0 \\
				0 & f(\calpha^{-1}(y)) & 0 &  \cdots & 0 & 0\\
			   \vdots & \vdots  & \vdots & \ddots &\vdots  &  \vdots \\
			   0 & 0 & 0 &  \cdots & f(\calpha^{2-k}(y)) & 0\\
				0 & 0 & 0 & \cdots & 0 &  f(\calpha^{1-k}(y))
\end{matrix}
\right )
\, , \,
\]
\[
w_\lambda = \left ( 
\begin{matrix} 0 & 1 & 0 &  \cdots & 0  \\
				0 & 0 & 1 &  \cdots & 0 \\
			   \vdots & \vdots  & \vdots & \ddots &\vdots   \\
			   0 & 0 & 0 &  \cdots & 1 \\
				\lambda & 0 & 0 & \cdots & 0 
\end{matrix}
\right )
\, .
\]
Up to unitary equivalence, this does not depend on the choice of $y$ in the orbit $[y]$, so it is a preimage for $([y],\lambda)$. This shows that $\Psi$ is surjective. Furthermore, any two preimages of $([y],\lambda)$ are unitarily equivalent. To see this, note that after conjugating by a suitable unitary, we can assume that the representation is of the standard form given by \ref{eq:standard-form-f} and \ref{eq:standard-form-u}. However the matrix in \ref{eq:standard-form-u} is unitarily equivalent via conjugation by a diagonal matrix to the matrix $v$ above. This shows that $\Psi$ is injective as well. 

Lastly, we note that $\Psi^{-1}$ is continuous. To see that, it suffices to consider a sufficiently small neighborhood of $([y],\lambda)$, for some $y \in Y_k$ and $\lambda \in \T$, which is homeomorphic to a neighborhood of $(y,\lambda) \in Y_k \times \T$. The map sending $(y,\lambda)$ to the pair $(\varphi_y, w_\lambda)$, thought of as a map from $Y_k \times \T$ to the set of representations of $A$ on $M_k$, is clearly continuous, and therefore the composition with the quotient map to $\mathrm{Prim}(B)$ is continuous as well. 

Since $Y$ is metrizable, by Proposition \ref{prop:Pears75}, we have $\mathrm{dim}(Y_k) = \mathrm{dim}(Y_k/\Z)$. (A more elementary way to see it is to observe that $Y_k \to Y_k/\Z$ is a covering map, so $Y_k/\Z$ can be written as a finite union of closed sets, each of which is homeomorphic to a closed subset of $Y_k$. The identity then follows from \cite[Chapter 3, Proposition 5.7]{Pears75}. This uses less advanced techniques from dimension theory.) Now, by Theorem \ref{thm:Pears75}, we have $\dim(Y_k) \leq \dim(Y)$. Therefore, 
$$
 \mathrm{dim} (\mathrm{Prim}_k(B)) =  \mathrm{dim}(Y_k)+\mathrm{dim}(\T) \leq \mathrm{dim}(Y)+1 \, . 
$$ 
By the main theorem of \cite{winter-subhomogeneous} (page 430 of that article), it follows that $$
\dr(C_0(Y) \rtimes_{\alpha} \Z) \leq \mathrm{dim}(Y)+1
\, ,
$$
 as required.
\end{proof} 

As we have remarked in the introduction, it is crucial that the upper bound we get does not depend on the minimal period of $\alpha$.

\section{Actions with uniformly compact orbits}
\label{section: uniformly compact orbits}

We provide here an alternative way to bound the nuclear dimension of a crossed product by a periodic action. This method has the advantage of working in the much more general setting of crossed products by locally compact Hausdorff second countable groups. Throughout the section, such a group will be denoted by $G$. Following the notations in Section \ref{section:prelim}, we let $X$ be a locally compact Hausdorff space and we let $\calpha: G \curvearrowright X$ be a continuous action by homeomorphisms.

\begin{defn}
 A continuous action $\calpha: G \curvearrowright X$ is said to have \emph{uniformly compact orbits} if there exists a compact subset $K \subset G$ such that for any $x \in X$, we have $G \cdot x = K \cdot x$.
\end{defn}

It is clear that any $\Z$-action arising from a periodic homeomorphism has uniformly compact orbits. 

\begin{lem}\label{lem:bddper-basic}
 Let $\calpha: G \curvearrowright X$ be a continuous action with uniformly compact orbits. Then:
 \begin{enumerate}
  \item 
  \label{lem:bddper-basic-1}
  Every orbit $G \cdot x$ is compact.
  \item 
  \label{lem:bddper-basic-2}
  For any $x \in X$ and for any neighborhood $U$ of the orbit $G \cdot x$, there exists a $G$-invariant open set $V$ such that $G \cdot x \subset V \subset \overline{V} \subset U$ and $\overline{V}$ is compact. 
  \item 
  \label{lem:bddper-basic-3}
  The quotient space $X / G $ is Hausdorff and locally compact.
  \item 
  \label{lem:bddper-basic-4}
  The quotient map $\pi \colon X \to X/G$ is proper.
 \end{enumerate}
\end{lem}

\begin{proof}
 (\ref{lem:bddper-basic-1}) is obvious. For (\ref{lem:bddper-basic-2}), the action extends trivially to the one point compactification, $X^+$, and the resulting action again has uniformly compact orbits, witnessed by the same compact subset $K \subset G$. Since $X^+$ is normal, there exists an open subset $W$ such that $G \cdot x \subseteq W \subseteq \overline{W} \subseteq U$. Now $X^+ \smallsetminus W$ is a compact set and thus so is $K \cdot (X^+ \smallsetminus W)$, which is equal to $G \cdot (X^+ \smallsetminus W)$ by the definition of having uniformly compact orbits. Define $V = X^+ \smallsetminus G \cdot (X^+ \smallsetminus W)$. One readily checks that $V$ satisfies the required properties. (\ref{lem:bddper-basic-3}) follows from (\ref{lem:bddper-basic-2}) by the definition of the quotient topology. As for (\ref{lem:bddper-basic-4}), let $C \subseteq X/G$ be a compact set. For any $x \in \pi^{-1}(C)$, find a $G$-invariant open set $V_x$ with compact closure so that $G \cdot x \subseteq V_x$. The collection $\{\pi(V_x)\}_{x \in X}$ forms an open cover of $C$, and therefore has a finite subcover, $\pi(V_{x_1}),\pi(V_{x_2}),\ldots,\pi(V_{x_n})$. Thus, $\pi^{-1}(C)$ is a closed subset of $\bigcup_{j=1}^n \overline{V_{x_j}}$, which is compact, and therefore $\pi^{-1}(C)$ is compact.  
\end{proof}

Using the notation of Lemma \ref{lem:bddper-basic} above, we have a homomorphism $\pi^* \colon C_0(X/G) \to C_0(X)$ given by $f \mapsto f \circ \pi$. Each element in $\pi^*(C_0(X/G))$ is $G$-invariant, and therefore defines an element in the center of the multiplier algebra of $C_0(X) \rtimes G$. This gives $C_0(X) \rtimes G$ the structure of a $C_0(X/G)$-algebra. We suppress the notation for $\pi^*$ in what follows.

If $Y \subseteq X$ is a closed subset, then we denote the restriction homomorphism from $C_0(X)$ to $C_0(Y)$ by $\mathrm{Res}_Y$. If $Y$ is $G$-invariant, then $\mathrm{Res}_Y$ is $G$-equivariant, and thus it induces a homomorphism $\mathrm{Res}_Y^{\rtimes G}: C_0(X) \rtimes G \to C_0(Y) \rtimes G$. In particular, if we set $Y = G \cdot x$ for some given $x$, then the kernel of $\mathrm{Res}_Y$ restricted to $C_0(X)$ is $C_0(X \smallsetminus \{G \cdot x\})$. Therefore, in the $C_0(X/G)$-algebra structure described above, the fiber over $G \cdot x$ is $C(G \cdot x) \rtimes G$. 

The following lemma is essentially taken from \cite[Lemma 3.1]{carrion}, with two minor differences. First, \cite[Lemma 3.1]{carrion} is stated for decomposition rank, whereas we need the analogous statement for nuclear dimension. However, the proof carries over essentially verbatim for nuclear dimension as well, and therefore we do not repeat it. Second, the statement there applies to $C(X)$-algebras where $X$ is assumed to be compact, and we need to use it for locally compact spaces. Again, the modification is trivial: any $C_0(X)$-algebra can be viewed as a $C(X^+)$-algebra, where the fiber at infinity is $0$. As the modifications needed are immediate, we do not repeat the proof.

\begin{Lemma}
\label{lem:bundle-estimate}
Let $Y$ be a locally compact Hausdorff second countable space, and let $A$ be a separable $C_0(Y)$-algebra. We denote by $A_y$ the fiber over $y$. Then:
 \[
  \dimnuc^{+1} ( A) \leq \dim^{+1} (Y) \cdot \sup_{y \in Y} \dimnuc^{+1} ( A_y ) 
 \]
and 
 \[
  \dr^{+1} ( A) \leq \dim^{+1} (Y) \cdot \sup_{y \in Y} \dr^{+1} ( A_y ) 
  \, .
 \]
 \qed
\end{Lemma}

We now apply this lemma to crossed products induced from actions with uniformly compact orbits. We denote the stabilizer group of a point $x \in X$ by $G_x = \{ g \in G ~|~ g \cdot x = x\}$.

\begin{thm}\label{thm:dimnuc-uniformly-compact-orbits}
 Let $\calpha: G \curvearrowright X$ be a continuous action with uniformly compact orbits. Then 
 \[
  \dimnuc^{+1} ( C_0(X) \rtimes G ) \leq \dim^{+1} (X / G) \cdot \sup_{x \in X} \dimnuc^{+1} ( C^*(G_x) ) 
 \]
and 
 \[
  \dr^{+1} ( C_0(X) \rtimes G ) \leq \dim^{+1} (X / G) \cdot \sup_{x \in X} \dr^{+1} ( C^*(G_x) ) \; . 
 \]
\end{thm}

\begin{proof}
Fix $x \in X$. Recall that by \cite[II.10.4.14]{blackadar-operator-algebras}, $C(G \cdot x) \rtimes G$ is strongly Morita equivalent to $C^*(G_x)$. Therefore, $\dimnuc (C(G \cdot x) \rtimes G) = \dimnuc(C^*(G_x))$ and $\dr (C(G \cdot x) \rtimes G) = \dr(C^*(G_x))$. The statement now follows immediately from Lemma \ref{lem:bundle-estimate} above.
\end{proof}

Finally, we relate the dimension of $X$ to that of its quotient by the action of $G$.
\begin{prop}\label{prop:dim-quotient-discrete}
 Let $\calpha: G \curvearrowright X$ be a continuous action with uniformly compact orbits. If $G$ is discrete, then $\dim(X / G ) = \dim(X)$. 
\end{prop}
\begin{proof}
 Since each orbit is finite, the quotient map from $X$ to $X / G$ is a finite-to-one map.
 The quotient map is furthermore open, since if $U$ is an open set in $X$, so is $G \cdot U$, and by the definition of the quotient topology, the image of $G \cdot U$ is open, and coincides with the image of $U$. The conclusion now follows directly from Proposition  \ref{prop:Pears75}.
\end{proof}

\section{Rokhlin towers for homeomorphisms without short orbits}
\label{section: long orbits}

We return to topological actions by $\Z$, and use the results from the appendix to construct Rokhlin-type towers, provided that there is lower bound on the lengths of the orbits which is large enough compared to the desired lengths of the towers. This is achieved through a refined version of the marker property, whose details are contained in Appendix \ref{sec:appendix}, by G\'abor Szab\'o. The following lemma is a special case of Lemma \ref{local marker lemma} for actions of $\Z$.

\begin{lem}
\label{lem:marker-Z}
 Let $X$ be a locally compact metrizable space with covering dimension at most $d$. Let $\calpha: \Z \curvearrowright X$ be an action. Suppose there is an $m \in \Z^+$ such that $| \Z \cdot x | > (d+1) (4m+1) $ for any $x \in X$. Then given any compact subset $ K \subset X $, there exists an open subset $Z \subset X$ such that 
 {\renewcommand{\theenumi}{\alph{enumi}}
 \begin{enumerate}
  \item \label{lem:marker-Z-a} $\alpha_{-m}(\overline{Z}), \ldots, \alpha_{0}(\overline{Z}) , \ldots, \alpha_{m}(\overline{Z})$ are pairwise disjoint;
  \item \label{lem:marker-Z-b} $\displaystyle K\subset \bigcup_{i = 1} ^{(d+1)(4m+1)} \alpha_i (Z)$.
 \end{enumerate}
 }
\end{lem}

\begin{proof}
 We set $G = \Z$, $F= [-m , m] \cap \Z$ and $g_l = (2m+1) + l (4m+1)$ for $l=0,1,\ldots,d$, whence $M = \bigcup_{l=0}^d g_l F^{-1} F = [1, (d+1)(4m+1)] \cap \Z$. The assumption that each orbit has length greater than $(d+1) (4m+1)$ implies that for any $x \in X$, the map $M^{-1} \to X$ given by $n \mapsto \calpha_n (x)$ is one-to-one, that is, $X(M^{-1}) = X$. Since $X$ is a locally compact metrizable space with covering dimension $d$, by Lemma \ref{local TSBP}, $(X,\calpha,\Z)$ has the $(M,d)$-small boundary property. Finally, applying Lemma \ref{local marker lemma} to $G$, $K$, $F$ and $M$, we obtain an open subset $Z \subset X$ with the desired properties.
\end{proof}

Notice that because of condition (\ref{lem:marker-Z-a}) in the previous lemma, the open cover $\{\alpha_i (Z)\}_{i = 1, \ldots (d+1)(4m+1)}$ of $K$ appearing in (\ref{lem:marker-Z-b}) may be split into $\left\lceil \frac{(d+1)(4m+1)}{2m+1} \right\rceil$ topological Rokhlin towers of length $(2m+1)$ in the sense of \cite[Section 2]{szabo}, possibly with some overlaps among the towers. Next, we construct a partition of unity subordinate to this open cover of $K$, which will play the role of $C^*$-algebraic Rokhlin towers in the sense of \cite{HWZ} (see Remark \ref{rmk:Rokhlin-towers} for further discussion). When we split $\{\alpha_i (Z)\}_{i = 1, \ldots (d+1)(4m+1)}$ into topological Rokhlin towers, it is advantageous to first do so with sufficiently large overlaps among the towers. These overlaps are controlled by a new parameter $k$ in the following lemma.

\begin{lem}\label{lem:tower-Z}
 Let $X$ be a locally compact metrizable space with covering dimension at most $d$. Fix $k, m \in \Z^+$ and a compact subset $ K \subset X $, and suppose $\calpha: \Z \curvearrowright X$ is an action such that $| \Z \cdot x | > (d+1) (4m+1) $ for any $x \in X$. 
 \begin{enumerate}
  \item \label{lem:tower-Z-1} If $m \geq (2d+3)k - \frac{d}{2} - 1$, then there exist open subsets $Z^{(0)}, \ldots, Z^{(2d+2)} \subset X$ such that 
  \begin{enumerate}
   \item 
   \label{lem:tower-Z-a}
   $\alpha_{-m}(\overline{Z}^{(l)} ), \ldots, \alpha_{0}(\overline{Z}^{(l)} ) , \ldots, \alpha_{m}(\overline{Z}^{(l)} )$ are pairwise disjoint for any $l \in \{0, \ldots, 2d+2\}$;
   \item 
   \label{lem:tower-Z-b}
   $\displaystyle K\subset \bigcup_{l = 0} ^{2d+2} \bigcup_{i = -(m-k)} ^{m-k} \alpha_i (Z^{(l)})$.
  \end{enumerate}
  \item \label{enum:tower-decay-Z} If $\varepsilon >0$ satisfies  $m \geq (2d+3) k \lceil \frac{1}{\varepsilon} \rceil - \frac{d}{2} - 1$, then there exist open subsets $Z^{(0)}, \ldots, Z^{(2d+2)} \subset X$ satisfying the above two conditions, as well as $\{ \mu^{(l)}_j \}_{l \in \{0, \ldots, 2d+2\} ; j \in \Z} \subset C_c(X)_{+, \leq 1}$ satisfying
  \begin{enumerate}
   \setcounter{enumii}{2}
   \item
   \label{lem:tower-Z-c}
    $ \mathrm{supp} ( \mu^{(l)}_j ) \subseteq 
    \begin{cases}
     \alpha_j (Z^{(l)}) , & \text{if}\ |j| \leq m \\
     \emptyset, & \text{if}\ |j| > m
    \end{cases}$
    for any $l \in \{0, \ldots, 2d+2 \}$;
   \item 
   \label{lem:tower-Z-d}
   $\displaystyle \sum_{l=0}^{2d+2} \sum_{j=-m}^{m} \mu_j^{(l)}(x) = 1$ for all $x \in K$;
   \item 
   \label{lem:tower-Z-e}
   $\left\| \mu_j^{(l)} \circ \calpha_i - \mu_{j-i}^{(l)} \right\| < \varepsilon$ for all $j \in \Z$, for all $i \in \Z \cap [-k, k]$ and for all $l \in \{0, \ldots, 2d+2 \}$. 
  \end{enumerate}
 \end{enumerate}
\end{lem}

\begin{proof}[Proof of Lemma \ref{lem:tower-Z}]
 To prove (\ref{lem:tower-Z-1}), we apply Lemma \ref{lem:marker-Z} to obtain $Z$, and then define, for $l = 0, \ldots, 2d+2$, 
 \[
  Z^{(l)} = \calpha_{ (2(m-k)+1) l + (m-k) + 1 } (Z) \; . 
 \]
 Condition (\ref{lem:tower-Z-a}) is satisfied because $\calpha_{-m}(\overline{Z}^{(l)} ), \ldots, \calpha_{0}(\overline{Z}^{(l)} ) , \ldots, \calpha_{m}(\overline{Z}^{(l)} ) $ are images of $\calpha_{-m}(\overline{Z}), \ldots, \calpha_{0}(\overline{Z}) , \ldots, \calpha_{m}(\overline{Z})$ under the homeomorphism $\calpha_{ (2(m-k)+1) l + (m-k) + 1 }$, and Condition (\ref{lem:tower-Z-b}) is verified by calculating the set of indices, as follows:
 \begin{align*}
  & \bigcup_{l = 0} ^{2d+2} \bigcup_{i = -(m-k)} ^{m-k} \calpha_i (Z^{(l)}) \\
  = & \bigcup \left\{ \calpha_i (Z) \ |\ i \in \Z \cap \bigcup_{l = 0} ^{2d+2} \bigg( (2(m-k)+1) l + (m-k) + 1 + \big[ -(m-k) , (m-k) \big] \bigg) \right\} \\
  = & \bigcup \left\{ \calpha_i (Z) \ |\ i \in \Z \cap \bigcup_{l = 0} ^{2d+2} \big[ (2(m-k)+1) l  + 1  , (2(m-k)+1) (l + 1) \big] \right\} \\
  = & \bigcup \left\{ \calpha_i (Z) \ |\ i \in \Z \cap \big[ 1  , (2(m-k)+1) (2d + 3) \big] \right\} \\
  \supset & \bigcup \left\{ \calpha_i (Z) \ |\ i \in \Z \cap \big[ 1  , (d+1)(4m+1) \big] \right\}  \supset  K  \; ,
 \end{align*}
 where the second-to-last step uses the inequality 
 \[
  (d+1)(4m+1) \leq (2d+3) (2(m-k) + 1) \; 
 \]
 which is a direct outcome of the definition of $m$. 
 
 In order to prove (\ref{enum:tower-decay-Z}), we set $k' = k  \lceil \frac{1}{\varepsilon} \rceil $ and $ K' = \bigcup_{i = -k'}^{k'} \calpha_i (K)$, and apply (\ref{lem:tower-Z-1}) with $k'$ and $K'$ in place of $k$ and $K$ to obtain open sets $Z^{(0)}, \ldots, Z^{(2d+2)} \subset X$ which satisfy (\ref{lem:tower-Z-a}) and (\ref{lem:tower-Z-b}) for $m$, $k'$ and $K'$ (and thus automatically for $k$ and $K$, too, as $k \leq k'$ and $K \subset K'$). Pick a partition of unity 
 $$
 \big \{ p^{(\infty)}\big \} \cup \big\{ p^{(l)}_j \ \mid \ l=0, \ldots, 2d+2 ; j \in \Z \cap [-(m-k'), m-k'] \big\} \subset C(X^+)_{+,\leq 1}
 $$
  such that $\mathrm{supp}(p^{(l)}_j) \subset  \calpha_{j}( Z^{(l)} )$ for all possible indices and $\mathrm{supp}(p^{(\infty)}) \subset X^+ \smallsetminus K'$. We set $p^{(l)}_j = 0 $ for all $l=0, \ldots, 2d+2$ and $j \in \Z \smallsetminus [-(m-k'), m-k']$. 
 
 Notice that the family $\{ p^{(l)}_j \}_{l=0, \ldots, 2d+2 ; j \in \Z}$ already satisfies (\ref{lem:tower-Z-c}) and (\ref{lem:tower-Z-d}). In order to produce elements $\{ \mu^{(l)}_j \}_{l=0, \ldots, 2d+2 ; j \in \Z}$ which also satisfy (\ref{lem:tower-Z-e}), we apply an averaging procedure to $\{ p^{(l)}_j \}_{j \in \Z}$ over a large F{\o}lner set, for each $l \in \{0, \ldots, 2d+2\}$, as follows. For any $l \in \{0, \ldots, 2d+2\}$ and $j \in \Z$, set 
 \[
  \mu_j^{(l)} = \frac{1}{2 k' +1} \sum_{i = -k'}^{k'} p^{(l)}_{j+i} \circ \calpha_{i} \; .
 \]
 We check that the family $\{ \mu^{(l)}_j \}_{l=0, \ldots, 2d+2 ; j \in \Z} \subset C_0(X)_{+, \leq 1}$ satisfies the desired conditions:
 \begin{enumerate}
  \item[(\ref{lem:tower-Z-c})] For any $j \in \{-m, \ldots, m\}$, we have
  \[
   \mathrm{supp} \left( \mu^{(l)}_j \right) \subseteq \bigcup_{i = -k'}^{k'} \mathrm{supp} \left( p^{(l)}_{j+i} \circ \calpha_{i} \right) \subseteq \bigcup_{i = -k'}^{k'}  \calpha_{-i} \left( \calpha_{j+i} \left( Z^{(l)} \right) \right) = \calpha_{j} \left( Z^{(l)} \right) \, ,
  \]
  while for any $j \in \Z \smallsetminus [-m, m]$ and  any $i \in \{-k', \ldots, k'\}$, we have $j+i \in \Z \smallsetminus [-(m-k'), m-k']$, whence 
  \[
   \mu_j^{(l)} = \frac{1}{2 k' +1} \sum_{i = -k'}^{k'} p^{(l)}_{j+i} \circ \calpha_{i} = 0 \; .
  \]
  \item[(\ref{lem:tower-Z-d})] For any $x \in K$,
  \begin{align*}
   \sum_{l=0}^{2d+2} \sum_{j=-m}^{m} \mu_j^{(l)}(x) & = \sum_{l=0}^{2d+2} \sum_{j\in \Z} \mu_j^{(l)}(x) & \\
   & = \frac{1}{2 k' +1} \sum_{l=0}^{2d+2} \sum_{j \in \Z}  \sum_{i = -k'}^{k'} p^{(l)}_{j+i} \left( \calpha_{i} (x) \right) & \\
   & = \frac{1}{2 k' +1} \sum_{i = -k'}^{k'} \sum_{l=0}^{2d+2} \sum_{j' \in \Z}  p^{(l)}_{j'} \left( \calpha_{i} (x) \right) \hspace{21mm} \Big[{\scriptstyle  j' = j+i} \Big] \\
   & = \frac{1}{2 k' +1} \sum_{i = -k'}^{k'} \left( 1 - p^{(\infty)} \left( \calpha_{i} (x) \right) \right) \hspace{10mm}  \Big[{\scriptstyle \text{partition\ of\ unity} } \Big] \\
   & = 1  \hspace{35mm}  \Big[ {\scriptstyle \calpha_{i} (x) \in K'\ \text{and} \ \mathrm{supp}(p^{(\infty)}) \subset X^+ \smallsetminus K' } \Big]  .
  \end{align*}
  \item[(\ref{lem:tower-Z-e})] For all $j \in \Z$, for all $i \in \Z \cap [0, k]$ and for all $l \in \{0, \ldots, 2d+2 \}$, 
  \begin{align*}
   & \left\| \mu_j^{(l)} \circ \calpha_i - \mu_{j-i}^{(l)} \right\| \\
   = & \ \frac{1}{2 k' +1} \left\| \sum_{i' = -k'}^{k'} p^{(l)}_{j+i'} \circ \calpha_{i'}  \circ \calpha_i - \sum_{i'' = -k'}^{k'} p^{(l)}_{j-i+i''} \circ \calpha_{i''}  \right\| \\
   = & \ \frac{1}{2 k' +1} \left\| \sum_{i' = -k'}^{k'} p^{(l)}_{j+i'} \circ \calpha_{i'+i} - \sum_{i' = -k'-i}^{k'-i} p^{(l)}_{j+i'} \circ \calpha_{i'+i}  \right\| \hspace{3mm}  \Big[{\scriptstyle i' = i'' - i \ \text{in\ the\ second\ sum} } \Big] \\
   = & \ \frac{1}{2 k' +1} \left\| \sum_{i' = k'-i+1}^{k'} p^{(l)}_{j+i'} \circ \calpha_{i'+i} - \sum_{i' = -k'-i}^{-k'-1} p^{(l)}_{j+i'} \circ \calpha_{i'+i}  \right\|  \hspace{1mm}  \Big[{\scriptstyle \text{cancelling\ identical\ terms} } \Big] \\
   \leq & \ \frac{1}{2 k' +1} ( (k' - (k' -i')) + (-k' - (-k'-i) ) \\
   = & \ \frac{2 i}{2 k' +1} \\
   \leq & \ \frac{2 k}{2 k \lceil \frac{1}{\varepsilon} \rceil +1} \\
   \leq & \ \varepsilon \, .
  \end{align*}
  The case of $i \in \Z \cap [-k, 0]$ is similar. 
 \end{enumerate}
 Therefore $\{ Z^{(l)}\}_{l=0, \ldots, 2d+2}$ and $\{ \mu^{(l)}_j \}_{l=0, \ldots, 2d+2 ; j \in \Z}$ satisfy the conditions in (\ref{enum:tower-decay-Z}). 
\end{proof}

\begin{Rmk}\label{rmk:Rokhlin-towers}
We recall from \cite{HWZ} that the definition of Rokhlin dimension (with single towers) requires the existence of positive elements $\{f_j^{(l)}\}_{j=0,1,\ldots,m-1 \, ; \, l=0,1,\ldots,d}$ such that $\alpha(f_j^{(l)}) \approx f_{j+1}^{(l)}$, for all applicable $j$ and $l$, and addition is taken modulo $m$ (so the tower is \emph{cyclic}). What we obtained here is elements $\mu_j^{(l)}$, where the index convention is $j=-m,\ldots,m$, and $\|\mu_m^{(l)}\|, \|\mu_{-m}^{(l)}\|<\eps$. Thus, in particular, we have $\alpha (\mu_{m}^{(l)}) \approx_{2\eps} \mu_{-m}^{(l)}$. One could refer to such Rokhlin towers as \emph{decaying Rokhlin towers}, and those could be used to define a variant of Rokhlin dimension. This is studied in \cite{SWZ} under the term \emph{amenability dimension}.
This kind of dimension would be comparable to the ordinary Rokhlin dimension, with a factor of $2$: any decaying Rokhlin tower is cyclic, and any cyclic tower can be made into two decaying Rokhlin towers by applying decay factors, as follows. Let $g:\Z \to [0,1]$ be the function such that $g(0) = 1$, $g(m+1) = g(-m-1) = 0$, $g$ is continued linearly in $[0,m+1] \cap \Z$ and in $[-m-1,0]$, and continued periodically to all of $\Z$. Suppose $\{f_j\}_{j=-m,\ldots,m}$ is a cyclic Rokhlin tower with tolerance $\eps$. For $j$'s outside $[-m,m] \cap \Z$, we extend the definition of $f_j$ periodically. Now, $\{g(j)f_j\}_{j \in [-m,m] \cap \Z}$ and $\{g(j+m)f_j\}_{j \in [0,2m] \cap \Z}$ are two decaying Rokhlin towers with tolerance $\eps+\frac{1}{2m}$. 

This technique is used in the proof of \cite[Theorem 4.1]{HWZ}, and is responsible for a factor of $2$ in the dimension estimates. Since in our setting we already obtain decaying Rokhlin towers, we avoid the need to repeat this trick in the proof of Theorem \ref{thm:estimate-dimnuc-Z} below. As a result, the bound we get here has a factor of $2$, as opposed to a factor of $4$ in \cite[Theorem 4.1]{HWZ}. 
\end{Rmk}

\section{The main result}
\begin{thm}\label{thm:estimate-dimnuc-Z}
 Let $X$ be a locally compact metrizable space and $\calpha \in \mathrm{Homeo}(X)$.  Then 
 \[
  \dimnuc(C_0(X) \rtimes_{\alpha} \Z) \leq 2 \big( \dim(X) \big)^2 + 6 \dim(X) +4 \; .
 \]
\end{thm}

The metrizability condition on $X$ can be removed. See Corollary \ref{cor:nonmetrizable}. 

\begin{proof}
 We assume that $X$ is finite dimensional, otherwise there is nothing to prove. Set $\dim (X) = d$. We need to show that $\dimnuc ( C_0(X) \rtimes_\alpha \Z ) \leq 2d^2+6d+4$. Recall that $\displaystyle  C_c(X) \rtimes_{\alpha, \mathrm{alg} } \Z = \Big\{ \sum_{\mathrm{finite\ sum}} f_i u_i \ \Big| \ f_i \in C_c(X), \quad \forall i \in \Z \Big\}$ is dense in $C_0(X) \rtimes_{\alpha} \Z$. Therefore, it is enough to verify that, given a finite subset $F \subset (C_c(X) \rtimes_{\alpha, \mathrm{alg} } \Z)_{\leq 1}$ and $\varepsilon > 0$, the condition of Lemma \ref{Lemma:finite-dimnuc} holds for $F$ and $(3d+7) \varepsilon$.
 
 Since $F$ is finite, there is $k \in \N$ such that 
 \[
  F \subset \left\{ \sum_{i = -k} ^k f_i u_i \in C_c(X) \rtimes_{\alpha, \mathrm{alg} } \Z \ \bigg| \ f_i \in C_c(X) , \quad \forall i \in \{ -k, \ldots, k \} \right\} \; .
 \]
 Fix $\eps'>0$ with the property that whenever $s,t \in [0,1]$ satisfy $|s-t|<\eps'$, we have $|\sqrt{s}-\sqrt{t}|<\frac{\eps}{2k+1}$. Define:
 \begin{align*}
  m & = (2d+3) k \left \lceil \frac{1}{\eps'} \right \rceil  \\ 
  N & = (d+1) (4m+1)  \\
  Y & = X_{\leq N} = \left\{ x \in X \ \big| \ [\Z : \mathrm{Stab}(x) ] \leq N \right\} = \left\{ x \in X \ \big| \ | \Z \cdot x | \leq N \right\} \, .
 \end{align*}
 For any $n \in \N$, the set $\{ x \in X \ | \ \calpha_n(x) = x \}$ is a closed $\calpha$-invariant subset of $X$. Therefore so is $X_{\leq N}$, being a finite union of closed invariant subsets. The action of $\calpha$ on $Y$ is periodic (with period $N!$, for instance). Therefore, by Proposition \ref{Prop:periodic points}, we have 
 $$
 \dr(C_0(Y) \rtimes_{\alpha} \Z) \leq d + 1 \, ,
 $$  
and in particular,
 $$
 \dimnuc(C_0(Y) \rtimes_{\alpha} \Z) \leq d + 1 .
 $$  
 Now, consider the short exact sequence 
 \[
  0 \to C_0(X \smallsetminus Y) \rtimes \Z \overset{\iota}{\longrightarrow} C_0(X ) \rtimes \Z \overset{\pi}{\longrightarrow} C_0(Y) \rtimes \Z  \to 0 \; .
 \]
 Since $\dimnuc^{+1} ( C_0(Y) \rtimes \Z ) \leq  d+2$, we may find a piecewise contractive $(d+2)$-decomposable approximation for $(\pi (F), \varepsilon)$:
 \[
  \xymatrix{\displaystyle
   C_0(Y) \rtimes \Z \ar[dr]_{\psi_Y = \bigoplus_{l=0} ^{d + 1} \psi_Y^{(l)} \ \ \ } \ar@{.>}[rr]^{\id} &  & C_0(Y) \rtimes \Z \\
   \displaystyle & A_{Y} = \bigoplus_{l=0} ^{d + 1} A_Y^{(l)} \ar[ur]_{\ \ \ \ \varphi_Y = \sum_{l=0} ^{d + 1} \varphi_Y^{(l)} } &
  }
 \]
 By Lemma \ref{Lemma:lifting-decoposable-maps}, we may lift $\varphi_Y$ to a piecewise contractive $(d+2)$-decomposable completely positive map 
 \[
  \widetilde{\varphi}_Y = \sum_{l=0} ^{d + 1} \widetilde{\varphi}_Y^{(l)} : A_Y  \to C_0(X ) \rtimes \Z \; .
 \]
 By \cite[Proposition 1.4]{winter-zacharias}, there exists $\delta > 0$ such that for any positive contraction $e \in C_0(X ) \rtimes \Z$, if $\left\| [ (1 - e), \widetilde{\varphi}_Y^{(l)} (a) ] \right\| \leq \delta \|a\| $ for any $a \in A_Y$ and for any $l \in \{0, \ldots, d+1\}$ then there are completely positive contractive order zero maps $\widehat{\varphi}_Y^{(l)} : A_Y \to C_0(X ) \rtimes \Z $ such that 
 \[
  \left\| \widehat{\varphi}_Y^{(l)} (a) - (1 - e)^\frac{1}{2} \widetilde{\varphi}_Y^{(l)} (a) (1 - e)^\frac{1}{2} \right\| \leq \varepsilon \| a \|
 \]
 for all $a \in A_Y$ and for all $l \in \{0, \ldots, d+1\}$. 
 
 By Lemma \ref{lem:quasicentral-approximate-unit}, there is a quasicentral approximate unit for $C_0(X \smallsetminus Y) \rtimes \Z \subseteq C_0(X ) \rtimes \Z$ which is contained in $C_c(X \smallsetminus Y)_{+, \leq 1}$. Thus, we may choose an element $e \in C_c(X \smallsetminus Y)_{+, \leq 1}$ which satisfies:
 \begin{enumerate}
  \item $\left\| [ (1 - e), \widetilde{\varphi}_Y^{(l)} (a) ] \right\| \leq \delta \|a\| $ for any $a \in A_Y$ and for any $l \in \{0, \ldots, d+1\}$
  \item $\left\| e^\frac{1}{2} \: b \: e^\frac{1}{2} + (1-e)^\frac{1}{2} \: b \: (1-e)^\frac{1}{2} - b \right\| \leq \varepsilon $ for any $b \in F$
  \item $\left\| (1-e)^\frac{1}{2} ( (\widetilde{\varphi}_Y \circ \psi_Y \circ \pi) (b) - b ) (1-e)^\frac{1}{2} \right\| \leq \varepsilon $ for any $b \in F$.
 \end{enumerate}
 Therefore, we can find maps $\widehat{\varphi}_Y^{(l)}$ as described in the previous paragraph, and sum them up to obtain a piecewise contractive $(d+2)$-decomposable completely positive map 
 \[
  \widehat{\varphi}_Y = \sum_{l=0} ^{d + 1} \widehat{\varphi}_Y^{(l)}  : A_Y \to C_0(X ) \rtimes \Z
 \]
 such that for any $a \in A_Y$,
 \[
  \left\| \widehat{\varphi}_Y (a) - (1 - e)^\frac{1}{2} \widetilde{\varphi}_Y (a) (1 - e)^\frac{1}{2} \right\| \leq (d+2) \varepsilon \| a \| \; .
 \]
 
 \begin{clm}\label{clm:estimate-1:thm:estimate-dimnuc-Z}
  The diagram
  \[
  \xymatrix{\displaystyle
   C_0(X) \rtimes \Z \ar[dr]_{\psi_Y \circ \pi \ \ } \ar@{.>}[rr]^{\id} &  & C_0(X) \rtimes \Z \\
   \displaystyle & A_{Y}  \ar[ur]_{\ \widehat{\varphi}_Y  } &
  }
 \]
 commutes on $(1 - e)^\frac{1}{2} F (1 - e)^\frac{1}{2}$ up to errors bounded by $(d+3)  \varepsilon$.
 \end{clm}
 Indeed, observe that for any $b \in F$, 
 \begin{align*}
  & \left\| (\widehat{\varphi}_Y \circ \psi_Y \circ \pi ) \left( (1 - e)^\frac{1}{2} b (1 - e)^\frac{1}{2} \right) - (1 - e)^\frac{1}{2} b (1 - e)^\frac{1}{2} \right\| & \\
  = & \  \left\| (\widehat{\varphi}_Y \circ \psi_Y \circ \pi ) \left( b \right) - (1 - e)^\frac{1}{2} b (1 - e)^\frac{1}{2} \right\| & \Big[\scriptscriptstyle{ \pi(e) = 0 }\Big] \\
  \leq & \ \left\| (\widehat{\varphi}_Y \circ \psi_Y \circ \pi ) ( b ) - (1 - e)^\frac{1}{2} \cdot ( \widetilde{\varphi}_Y \circ \psi_Y \circ \pi ) ( b ) \cdot (1 - e)^\frac{1}{2} \right\| \\
  & + \left\| (1 - e)^\frac{1}{2} \big( (\widetilde{\varphi}_Y \circ \psi_Y \circ \pi ) ( b ) - b \big) (1 - e)^\frac{1}{2} \right\| \\
  \leq & \ (d+2) \varepsilon \left\| (\psi_Y \circ \pi ) ( b ) \right\| + \varepsilon \\
  \leq & \ (d+3)  \varepsilon .  & \Big[\scriptstyle{\|b\| \leq 1 }\Big] 
 \end{align*}
 This proves the claim.

 The next step is to find an approximation for $ e^\frac{1}{2} F e^\frac{1}{2} $. Define 
 $$
 K = \mathrm{supp}(e) \subset X \smallsetminus Y
 \, .
 $$
Our choice of $m$ allows us to apply Lemma \ref{lem:tower-Z}(\ref{enum:tower-decay-Z}) to $X \smallsetminus Y$, the compact subset $K$ and the parameters $k$, $m$ and $\varepsilon'$. This produces open subsets $Z^{(0)}, \ldots, Z^{(2d+2)} \subset X$ and functions $\{ \mu^{(l)}_j \}_{l=0, \ldots, 2d+2 ; j \in \Z} \subset C_c(X)_{+, \leq 1}$ satisfying the conditions of the lemma (with $\eps'$ in place of $\eps$).  
 
 Consider the regular representation of $C_0(X) \rtimes_{\alpha} \Z$ on the Hilbert module $E=\ell^2(\Z,C_0(X))$. That is, we embed $C_0(X)$ in $B(E)$ by $(f \cdot \xi) (n) = (f \circ \calpha_n)\xi (n)$, where $f \in  C_0(X)$ and $\xi \in E$, and we identify the canonical unitary $u$ with the bilateral shift operator in $B(E)$.

 Let $Q$  be the projection onto the subspace $\ell^2(\{-m, \ldots ,m\}, C_0(X) ) \subseteq E$. So, $b \mapsto  Q b Q$ is a completely positive map from  $C_0(X)  \rtimes_\alpha \Z $ to $C_0(X) \otimes B \big(\ell^2(\{-m, \ldots ,m\})\big)$. Let 
 \[
  \left\{ E_{i j} \right\}_{i,j \in \{-m, \ldots, m\}} \in B \big(\ell^2(\{-m, \ldots ,m\})\big)
 \]
 be the canonical linear basis consisting of matrices with $1$ at the $(i,j)$-th entry and $0$ elsewhere. 
 
 For any $l \in \{0, \ldots, 2d+2\}$, we define 
 \[
  A_{X \smallsetminus Y}^{(l)} = C_0(Z^{(l)}) \otimes B \big(\ell^2(\{-m, \ldots ,m\})\big) \cong M_{2m + 1} (C_0(Z^{(l)})) \; .
 \]
 By construction, we have $\mathrm{supp}(\mu^{(l)}_j \circ \calpha_j) \subset Z^{(l)}$ for any $j\in \Z$. Thus, we may define
 \[
  \mu^{(l)} = \sum_{j = -m}^{m} \left( \mu^{(l)}_j \circ \alpha_j \right) \otimes E_{jj} \in \left( A_{X \smallsetminus Y}^{(l)} \right)_{+, \leq 1} \; .
 \]
 Since $C_0(Z^{(l)})$ is an ideal in $C_0(X)$, it follows that $A_{X \smallsetminus Y}^{(l)}$ is an ideal in $C_0(X) \otimes B \big(\ell^2(\{-m, \ldots ,m\})\big)$. We can thus define completely positive contractions
 \[
  \psi_{X \smallsetminus Y}^{(l)} : C_0(X) \rtimes_{\alpha} \Z \to A_{X \smallsetminus Y}^{(l)} 
 \]
 by 
 $$
 \psi_{X \smallsetminus Y}^{(l)}(b) = \sqrt{\mu^{(l)}} Q b Q \sqrt{\mu^{(l)}} \; .
$$
 Next, we define a linear map
 \[
  \varphi_{X \smallsetminus Y}^{(l)} : A_{X \smallsetminus Y}^{(l)} \to C_0(X) \rtimes_{\alpha} \Z  
 \]
 by setting, for all $f \in C_0(Z^{(l)})$ and for all $i,j \in \{-m, \ldots, m\}$, 
 \[
  \varphi_{X \smallsetminus Y}^{(l)} ( f \otimes E_{ij} ) =  (f \circ \calpha_{-i}) u_{i-j} \; .
 \]
 We claim that $\varphi_{X \smallsetminus Y}^{(l)}$ is a $*$-homomorphism. Indeed, for any $f, f' \in C_0(Z^{(l)})$ and $i,j , i',j' \in \{-m, \ldots, m\}$, we have
 \begin{align*}
  & \varphi_{X \smallsetminus Y}^{(l)} ( ( f \otimes E_{ij} )^* ) &=\ & \varphi_{X \smallsetminus Y}^{(l)} ( \overline{f} \otimes E_{ji} ) &=\ & (\overline{f} \circ \calpha_{-j}) u_{j-i}  & \\
  & &=\ & u_{j-i} (\overline{f} \circ \calpha_{-i}) &=\ & \left( (f \circ \calpha_{-i}) u_{i-j} \right)^* & = \varphi_{X \smallsetminus Y}^{(l)} ( ( f \otimes E_{ij} ) )^*
 \end{align*}
 and, using the fact that if $j \not= i'$ then 
  \[
   \mathrm{supp}(f \circ \alpha_{-j}) \cap \mathrm{supp}(f \circ \alpha_{-i'}) \subset \alpha_{j}(Z^{(l)}) \cap \alpha_{i'}(Z^{(l)}) = \emptyset \; ,
  \]
  we check that
 \begin{align*}
  \varphi_{X \smallsetminus Y}^{(l)} ( f \otimes E_{ij} ) \cdot \varphi_{X \smallsetminus Y}^{(l)} ( f' \otimes E_{i'j'}) & = (f \circ \calpha_{-i}) u_{i-j} (f' \circ \calpha_{-i'}) u_{i'-j'} \\
  & = \Big( \big( (f \circ \calpha_{-j}) \cdot (f' \circ \calpha_{ - i'}) \big) \circ \calpha_{j-i} \Big) u_{i-j} u_{i'-j'} \\
  & = \begin{cases}
       \big ( (f \cdot f') \circ \calpha_{-i} \big) u_{i-j'} & , \ \text{if}\ j = i' \\
       0 & , \ \text{if}\  j \not= i'
      \end{cases} \\
  & = \varphi_{X \smallsetminus Y}^{(l)} \big( ( f \otimes E_{ij} ) \cdot ( f' \otimes E_{i'j'}) \big) \; .
 \end{align*}
 This proves the claim. 
 
 \begin{clm}\label{clm:estimate-2:thm:estimate-dimnuc-Z}
  The diagram
  \[
  \xymatrix{\displaystyle
   C_0(X) \rtimes \Z \ar[dr]_{\psi_{X \smallsetminus Y} = \bigoplus_{l=0} ^{2d + 2} \psi_{X \smallsetminus Y}^{(l)} \ \ \ } \ar@{.>}[rr]^{\id} &  & C_0(X) \rtimes \Z \\
   \displaystyle & A_{X \smallsetminus Y} = \bigoplus_{l=0} ^{2d + 2} A_{X \smallsetminus Y}^{(l)}  \ar[ur]_{\ \ \  \varphi_{X \smallsetminus Y} = \sum_{l=0} ^{2d + 2} \varphi_{X \smallsetminus Y}^{(l)}  } &
  }
 \]
 commutes on $e^\frac{1}{2} F e^\frac{1}{2}$ up to errors bounded by $(2d+3)\varepsilon $.
 \end{clm}
 We will prove this claim after we complete the main body of the proof. \\

 To finish the proof,
set $\psi \colon C_0(X) \rtimes \Z \to  A_{Y} \oplus A_{X \smallsetminus Y}$ to be   
$$\psi (b) = ( \psi_{Y} \circ \pi ) \left( (1 - e)^\frac{1}{2} b (1 - e)^\frac{1}{2} \right) \oplus \psi_{X \smallsetminus Y} \left( e^\frac{1}{2} b e^\frac{1}{2} \right)
\, ,
$$
 and 
 consider the diagram
 \[
  \xymatrix{\displaystyle
   C_0(X) \rtimes \Z \ar[dr]_{\psi  \ } \ar@{.>}[rr]^{\id} &  & C_0(X) \rtimes \Z \\
   \displaystyle & A_{Y} \oplus A_{X \smallsetminus Y}  \ar[ur]_{\ \ \  \varphi = \widehat{\varphi}_{Y} + \varphi_{X \smallsetminus Y}  } &
  }
 \]
 
 Observe that:
 \begin{enumerate}
  \item $\psi$ is completely positive and contractive.
  \item $\varphi_{Y}$ is a sum of $(d + 2)$-many order zero contractions, and $\varphi_{X \smallsetminus Y}$ is a sum of $(2d+3)$-many $*$-homomorphisms (which, in particular, are order zero contractions).
  \item For all $b \in F$, we compute, using the bounds given by Claims \ref{clm:estimate-1:thm:estimate-dimnuc-Z} and \ref{clm:estimate-2:thm:estimate-dimnuc-Z} and using the properties of $e$:
  \begin{align*}
   & \left\|\varphi(\psi(b)) - b \right\| \\
   \leq & \ \left\| ( \widehat{\varphi}_Y \circ \psi_Y \circ \pi ) \left( (1 - e)^\frac{1}{2} b (1 - e)^\frac{1}{2} \right) - (1 - e)^\frac{1}{2} b (1 - e)^\frac{1}{2} \right\| \\
   & + \left\|  ( \varphi_{X \smallsetminus Y} \circ \psi_{X \smallsetminus Y} ) \left( e^\frac{1}{2} b e^\frac{1}{2} \right) -  e^\frac{1}{2} b e^\frac{1}{2} \right\| \\
   & + \left\| (1 - e)^\frac{1}{2} b (1 - e)^\frac{1}{2} + e^\frac{1}{2} b e^\frac{1}{2} - b \right\| \\
   \leq & \ (d+3) \varepsilon + (2d+3) \varepsilon + \varepsilon \\
   = & \ (3d+7) \varepsilon \, .
  \end{align*}
  \item $\dimnuc(A_{Y}) = 0$ and 
  \[
   \dimnuc(A_{X \smallsetminus Y}) = \max_{l \in \{0, \ldots, 2d+2\} } \dimnuc(C_0(Z^{(l)})) \leq \dimnuc(C_0(X)) = d \, ,
  \]
  and thus 
  \begin{align*}
   & (\dimnuc(A_{Y}) + 1) (d+2) + ( \dimnuc(A_{X \smallsetminus Y}) + 1) (2d+3) \\
   \leq & \ (d+2) + (d+1)(2d+3) \\
   = & 2d^2+6d+5 \; .
  \end{align*}
 \end{enumerate}
 Therefore the result follows from Lemma \ref{Lemma:finite-dimnuc}. 
\end{proof}

\begin{proof}[Proof of Claim \ref{clm:estimate-2:thm:estimate-dimnuc-Z}]
 First we observe that for any $l \in \{0 , \ldots, 2d+2 \}$, for any $ f \in C_0(X) $ and for any $i \in \Z$, 
 \begin{align*}
  & \psi_{X \smallsetminus Y}^{(l)} ( f u_i ) \\
  = & \ \sqrt{\mu^{(l)}} Q f u_i Q \sqrt{\mu^{(l)}} \\
  = & \ \sqrt{\mu^{(l)}} \left( \sum_{ \begin{aligned}
   j \in \{-m, \ldots , m \} \cap \\
   \{-m +i, \ldots , m +i \}
  \end{aligned} }  (f \circ \calpha_j ) \otimes E_{j, j-i} \right) \sqrt{\mu^{(l)}} \\
  = & \ \sum_{ \begin{aligned}
   j \in \{-m, \ldots , m \} \cap \\
   \{-m +i, \ldots , m +i \}
  \end{aligned} } \left( \left( f \cdot \sqrt{\mu^{(l)}_{j} } \cdot  \left( \sqrt{\mu^{(l)}_{j-i}} \circ \calpha_{-i} \right) \right) \circ \calpha_j \right) \otimes E_{j, j-i}  \; .
 \end{align*}
 Using the fact that $\mu^{(l)}_j = 0$ whenever $|j| > m$, we may enlarge the domain of summation to $\Z$, with the understanding that the summand is $0$ whenever $\mu^{(l)}_j = 0$ or $\mu^{(l)}_{j-i} = 0$. Thus 
 \begin{align*}
  \varphi_{X \smallsetminus Y}^{(l)} \left( \psi_{X \smallsetminus Y}^{(l)} ( f u_i ) \right) = & \ \sum_{ j \in \Z } \left( \left( f \cdot \sqrt{\mu^{(l)}_{j} } \cdot  \left( \sqrt{\mu^{(l)}_{j-i}} \circ \calpha_{-i} \right) \right) \circ \calpha_j \circ \calpha_{-j} \right) u_{j - (j-i)} \\
  = & \ \left( \sum_{ j \in \Z } \sqrt{\mu^{(l)}_{j} } \cdot  \left( \sqrt{\mu^{(l)}_{j-i}} \circ \calpha_{-i} \right) \right) \cdot f u_i \; .
 \end{align*}
 If, in addition, $\mathrm{supp}(f) \subset K$, then
 \begin{align*}
  & \left( \sum_{l=0} ^{2d + 2} \varphi_{X \smallsetminus Y}^{(l)} \left( \psi_{X \smallsetminus Y}^{(l)} ( f u_i ) \right) \right) - f u_i \\
  = & \ \left( \sum_{l=0} ^{2d + 2} \sum_{ j \in \Z } \sqrt{\mu^{(l)}_{j} } \cdot  \left( \sqrt{\mu^{(l)}_{j-i}} \circ \calpha_{-i} \right) \right) \cdot f u_i - \left( \sum_{l=0} ^{2d + 2} \sum_{ j \in \Z } {\mu^{(l)}_{j} }  \right) \cdot f u_i \\
  = & \ \left( \sum_{l=0} ^{2d + 2} \sum_{ j \in \Z } \sqrt{\mu^{(l)}_{j} } \cdot  \left( \sqrt{\mu^{(l)}_{j-i}} \circ \calpha_{-i} - \sqrt{\mu^{(l)}_{j} } \right) \right) \cdot f u_i
 \end{align*}
 Recall that $\left\| \mu_{j-i}^{(l)} \circ \calpha_{-i} - \mu_{j}^{(l)} \right\| < \eps'$ for all $j \in \Z$, for all $i \in \Z \cap [-k, k]$ and for all $l \in \{0, \ldots, 2d+2 \}$, and $\eps'$ was chosen so that if $s,t \in [0,1]$ satisfy $|s-t|<\eps'$ then $|\sqrt{s}-\sqrt{t}|<\frac{\eps}{2k+1}$. Thus, it follows that $\left\| \sqrt{ \mu_{j-i}^{(l)} } \circ \calpha_{-i} - \sqrt{ \mu_{j}^{(l)} } \right\| < \frac{\varepsilon}{2k+1}$ for all $j \in \Z$, for all $i \in \Z \cap [-k, k]$ and for all $l \in \{0, \ldots, 2d+2 \}$. Consequently,
 \begin{align*}
  & \left\| \left( \sum_{l=0} ^{2d + 2} \varphi_{X \smallsetminus Y}^{(l)} \left( \psi_{X \smallsetminus Y}^{(l)} ( f u_i ) \right) \right) - f u_i \right\| \\
  = & \ \left\| \left( \sum_{l=0} ^{2d + 2} \sum_{ j \in \Z } \sqrt{\mu^{(l)}_{j} } \cdot  \left( \sqrt{\mu^{(l)}_{j-i}} \circ \calpha_{-i} - \sqrt{\mu^{(l)}_{j} } \right) \right) \cdot f u_i \right\| \\
  \leq & \ \sup_{x\in K} \left\{ \sum_{l=0} ^{2d + 2} \sum_{ j \in \Z } \sqrt{\mu^{(l)}_{j} (x) } \cdot  \left| \sqrt{\mu^{(l)}_{j-i} ( \calpha_{-i} (x) ) } - \sqrt{\mu^{(l)}_{j} (x) } \right| \right\} \cdot \|f\| \|u_i\| \\
  \leq & \ \sum_{l=0} ^{2d + 2} \frac{\varepsilon}{2k+1} \cdot \|f\|  
  \quad \quad \quad \Big[ \scriptstyle{ \forall l \in \{0, \ldots, 2d+2\}, \mathrm{supp}(\mu^{(l)}_{-m}) , \ldots, \mathrm{supp}(\mu^{(l)}_{m}) \text{\ are\ mutually\ disjoint} } \Big]\\
  = & \  (2d+3) \cdot \frac{\varepsilon}{2k+1} \cdot \|f\| \; .
 \end{align*}
 Now fix any $b = \sum_{j=-k}^k f_j u_j \in F$. Note that for any $j \in \{-k , \ldots , k\}$, we have $f_j = E(b u_{-j})$, where $E: C_0(X) \rtimes \Z \to C_0(X)$ is the standard conditional expectation, whence $\|f_j\| \leq \|b\| \leq 1$. Since $\mathrm{supp}(e) \subset K$ by our construction, we have
 \begin{align*}
  & \left\| \left( \sum_{l=0} ^{2d + 2} \varphi_{X \smallsetminus Y}^{(l)} \left( \psi_{X \smallsetminus Y}^{(l)} ( e^\frac{1}{2} b e^\frac{1}{2} ) \right) \right) - e^\frac{1}{2} b e^\frac{1}{2} \right\| \\
  \leq & \ \sum_{j=-k}^k \left\| \left( \sum_{l=0} ^{2d + 2} \varphi_{X \smallsetminus Y}^{(l)} \left( \psi_{X \smallsetminus Y}^{(l)} \left( e^\frac{1}{2} f_j \left(e^\frac{1}{2} \circ \calpha_{-j} \right) u_j \right) \right) \right) - e^\frac{1}{2} f_j \left(e^\frac{1}{2} \circ \calpha_{-j} \right) u_j \right\| \\
  \leq & \ \sum_{j=-k}^k (2d+3) \cdot \frac{\varepsilon}{2k+1} \cdot \left\| e^\frac{1}{2} f_j \left(e^\frac{1}{2} \circ \calpha_{-j} \right) \right\| \\
  \leq & \ \varepsilon (2d+3)  \; .
 \end{align*}
 This proves the claim.
\end{proof}

The following corollary addresses the case in which the space is not separable. The reason for the difference in the statement is to avoid the issue of defining $\dim(X)$ when $X$ is not metrizable. The appropriate definition in our case is $\dimnuc(C_0(X))$.
\begin{cor}\label{cor:nonmetrizable}
 Let $X$ be a locally compact Hausdorff space and $\calpha \in \mathrm{Homeo}(X)$.  Then 
 \[
  \dimnuc(C_0(X) \rtimes_{\alpha} \Z) \leq 2 \big( \dimnuc(C_0(X)) \big)^2 + 6 \dimnuc(C_0(X)) + 4 \; .
 \]
\end{cor}

\begin{proof}
 We reduce the situation to the metrizable case as follows. By Lemma \ref{lem:separable-dimnuc}, $C_0(X)$ is the union of $\Z$-invariant separable $C^*$-subalgebras $C_0(Y)$ with $\dim(Y) \leq \dimnuc (C_0(X))$. Therefore if we let $I$ be the net of all $\Z$-invariant separable $C^*$-subalgebras of $C_0(X)$ with nuclear dimension no more than $\dimnuc(C_0(X))$, ordered by inclusion, then since the spectrum of a commutative separable $C^*$-algebra is metrizable, we have
 \begin{align*}
  \dimnuc^{+1}(C_0(X) \rtimes \Z) \leq & \liminf_{C_0(Y) \in I} \dimnuc^{+1}(C_0(Y) \rtimes \Z) 
 \end{align*}
 and hence the statement follows immediately from Theorem \ref{thm:estimate-dimnuc-Z}.
\end{proof}

\begin{rmk}
 One can distill from the proof of Theorem \ref{thm:estimate-dimnuc-Z} the following more general statement, which can be seen as a refinement of \cite[Proposition 2.9]{winter-zacharias}: 
 
 Let $A$ be a $C^*$-algebra and $d_1, d_2 \in \N$. Suppose for any finite subset $F \subset A$ and any $\varepsilon > 0$, there exists an ideal $I \lhd A$, a quasicentral approximate unit $\{e_\lambda\}_{\lambda \in \Lambda} \subset I$ and $\lambda_0\in \Lambda$ such that $\dimnuc(A / I) \leq d_1$ and for all $\lambda \geq \lambda_0$, there exists a $d_2$-decomposable approximation for $e_{\lambda}^{\frac{1}{2}} F e_{\lambda}^{\frac{1}{2}}$ with tolerance $\varepsilon$. Then we have $\dimnuc(A) \leq d_1 + d_2 + 1$.
\end{rmk}

\appendix
\section{}\label{sec:appendix}
\begin{center}
By G\'abor Szab\'o
\footnote{Westf{\"a}lische Wilhelms-Universit{\"a}t, Fachbereich Mathematik, \\ Einsteinstrasse 62, 48149 M{\"u}nster, Germany\\ \textit{E-mail address:} \texttt{gabor.szabo@uni-muenster.de}}
\end{center}

In this note, we establish a technical result of topological nature, Lemma \ref{local marker lemma}, that is needed to prove Theorem \ref{thm:estimate-dimnuc-Z}. (The lemma is stated for actions of arbitrary groups, however only the special case of $\Z$-actions is needed for Theorem \ref{thm:estimate-dimnuc-Z}.)
To be more specific, we need to generalize the topological results from \cite[Section 3]{szabo} that have led to the marker property \cite[4.3, 4.4]{szabo} for free actions on finite-dimensional spaces. For aperiodic homeomorphisms, this has been proved previously by Gutman in \cite{gutman}, building on a technical result by Lindenstrauss from \cite{Lindenstrauss95}.

The way that we generalize the topological results from \cite{szabo} is twofold. First, we do not restrict our attention to topological dynamical systems on compact spaces, but consider the locally compact case. Although it was remarked in \cite[5.4]{szabo} how one could modify the proofs to cover the locally compact case, this was never carried out explicitly or in detail. Secondly, we do not focus only on free actions, or actions having other weaker global freeness properties such as \cite[3.4]{szabo}. Instead, we consider sufficiently good local freeness properties of group actions (with respect to certain finite subsets of the acting group), and deduce a weaker, localized marker-type property, see Lemma \ref{local marker lemma}. Otherwise, the approach is almost identical with that in \cite{szabo}. The reader is warned that what we refer to as a `localized marker-type property' in this note is not related to Gutman's notion of the local marker property as considered in \cite{Gutman-cubical}.

This general perspective, which takes into account local information about a group action rather than global, is crucial for the proof of Theorem \ref{thm:estimate-dimnuc-Z}. 

The results of this note were obtained during the author's doctoral studies and are part of his dissertation \cite{szabo15}.

\begin{defi}[cf.~{\cite[3.1]{Lindenstrauss95}} and {\cite[3.1]{szabo}}] \label{disjointness} 
Let $X$ be a locally compact metric space, $G$ a discrete group and $\alpha: G\curvearrowright X$ an action. Let $M\subset G$ be a subset and $k\in\IN$ be some natural number. We say that a set $E\subset X$ is $(M,k)$-disjoint, if for all distinct elements $\gamma(0),\dots,\gamma(k)\in M$ we have
\[
\alpha_{\gamma(0)}(E)\cap\dots\cap\alpha_{\gamma(k)}(E)=\emptyset.
\]
\end{defi}

\begin{lemma}[cf.~{\cite[3.7]{szabo}}] \label{disjointness neighbourhood} 
Let $X$ be a locally compact metric space with a group action $\alpha: G\curvearrowright X$.
Let $F\fin G$ be a finite subset and $n\in\IN$ a natural number. If a compact subset $E\subset X$ is $(F,n)$-disjoint, then there exists an open, relatively compact neighbourhood $V$ of $E$ such that $\quer{V}$ is $(F,n)$-disjoint.
\end{lemma}
\begin{proof}
Note that for all $S\subset F$ with $n=|S|$, we have
\[\emptyset=\bigcap_{\gamma\in S} \alpha_\gamma(E)= \bigcap_{\gamma\in S} \alpha_\gamma\Bigl( \bigcap_{\eps>0} \quer{B}_\eps(E) \Bigl)
= \bigcap_{\eps>0} \left( \bigcap_{\gamma\in S} \alpha_\gamma(\quer{B}_\eps(E)) \right)\]
By compactness, there must exist some $\eps(S)>0$ such that $\quer{B}_{\eps(S)}(E)$ is compact and 
\[
\displaystyle \bigcap_{\gamma\in S} \alpha_\gamma(\quer{B}_{\eps(S)}(E))=\emptyset.
\]
If we set $\eps=\min\set{ \eps(S) ~|~ S\subset F, n=|S|}$, then $V=B_\eps(E)$ is a relatively compact, open neighbourhood of $E$ whose closure is $(F,n)$-disjoint.
\end{proof}

\begin{defi}[cf.~{\cite[3.2]{Lindenstrauss95}} and {\cite[3.2]{szabo}}]
Let $X$ be a locally compact metric space, $G$ a group and $\alpha: G\curvearrowright X$ an action. Let $d\in\IN$ be a natural number and $M\fin G$ a finite subset. We say that the topological dynamical system $(X,\alpha,G)$ has the $(M,d)$-small boundary property, if whenever $K\subset X$ is compact and $V\supset K$ is open, we can find a relatively compact, open set $U$ with $K\subset U\subset V$ such that $\del U$ is $(M,d)$-disjoint.
\end{defi}

In order to make general statements for actions on finite-dimensional spaces, we naturally need to apply dimension theory for topological spaces.
More specifically, we shall now record some well-known facts about properties of covering dimension, which we will refer to throughout this section. These statements come up in \cite[Section 3]{Lindenstrauss95} and \cite[Section 3]{szabo}, but a detailed treatment can be found in \cite{Engelking}, see in particular \cite[4.1.5, 4.1.7, 4.1.9, 4.1.14, 4.1.16]{Engelking}.
All spaces in question are assumed to be separable metric spaces.
{\renewcommand{\theenumi}{D\arabic{enumi}}
\begin{enumerate}
 \item \label{D1} $A\subset B$ implies $\dim(A)\leq \dim(B)$.
 \item \label{D2} If $\set{B_i}_{i\in\IN}$ is a countable family of closed sets in $A$ with $\dim(B_i)\leq k$, then $\dim(\bigcup B_i)\leq k$.
 \item \label{D3} Let $E \subset A$ be a zero dimensional subset and $x\in U\subset A$ a point with an open neighbourhood. Then there exists some open set $U'\subset A$ with $x\in U'\subset U$ such that $\del U' \cap E=\emptyset$.
 \item \label{D4} If $A\neq\emptyset$, there exists a zero dimensional $F_\sigma$-set $E\subset A$ such that $\dim(A\setminus E)=\dim(A)-1$.
 \item \label{D5} Any countable union of $k$-dimensional $F_\sigma$-sets is a $k$-dimensional $F_\sigma$-set.
\end{enumerate}
}

\begin{lemma}[cf.~{\cite[3.3]{szabo}}] \label{Startlemma} 
Let $X$ be a locally compact metric space. Let $K\subset X$ be compact and $V\subset X$ an open neighbourhood of $K$. Let $E\subset X$ be a zero dimensional subset. Then there exists a relatively compact, open set $U$ with $K\subset U\subset \quer{U}\subset V$ such that $\del U\cap E=\emptyset$.
\end{lemma}
\begin{proof}
Clearly $\del K$ is compact. For $x\in\del K$, apply \eqref{D3} and find relatively compact, open neighbourhoods $x\in B_x\subset\quer{B}_x\subset V$ such that $\del B_x \cap E=\emptyset$. Choose a finite cover $\del K\subset \bigcup_{i=1}^M B_i$ of such neighbourhoods and set $U=K\cup\bigcup_{i=1}^M B_i$. It is now immediate that $U$ is relatively compact with $\quer{U}\subset V$ and that $\del U \subset \bigcup_{i=1}^M \del B_i$, so we have indeed $\del U \cap E=\emptyset$.
\end{proof}

The following is an ad-hoc notational convention for this Appendix that makes it easier to keep track of local freeness properties of group actions.

\begin{defi} \label{M-free}
Let $X$ be a locally compact metric space, $G$ a group and $\alpha: G\curvearrowright X$ an action. Let $M\fin G$ be a finite subset. 
We define
\[
X(M) = \set{x\in X ~|~ \text{the map}~[M\ni g\mapsto \alpha_g(x)]~\text{is injective} }.
\]
By continuity, $X(M)$ is an open subset of $X$.
The action $\alpha$ is then free if and only if one has $X(M)=X$ for every $M\fin G$.
\end{defi}

\begin{defi}[following {\cite[Section 3]{Kulesza95}} and {\cite[3.4]{Lindenstrauss95}}] \label{genposition} 
Let $X$ be a metric space of finite covering dimension $n$. A family
$\CB$ of subsets in $X$ is in general position, if for all finite subsets $S\fin\CB$ we have
\[
\dim(\bigcap S) \leq \max(-1,n-|S|).
\]
\end{defi}

\begin{lemma}[cf.~{\cite[3.6]{szabo}}] \label{local TSBP} 
Let $X$ be a locally compact metric space with finite covering dimension $d$, let $G$ be a group and $\alpha: G\curvearrowright X$ an action. Let $M\fin G$ be finite subset.
Let $K\subset X(M)$ be compact and let $V\subset X(M)$ be an open neighbourhood of $K$. Then there exists a relatively compact, open set $U$ with $K\subset U\subset \quer{U}\subset V$ such that the family $\set{\alpha_\gamma(\del U)}_{\gamma\in M}$ is in general position in $X$. In particular, the set $\del U$ is $(M,d)$-disjoint.
\end{lemma}
\begin{proof}
First observe the following. If $E\subset X$ is any subset such that $\CB=\set{\alpha_\gamma(\del E)}_{\gamma\in M}$ is in general position, then the set $E$ is automatically $(M,d)$-disjoint. This is because by definition, the intersection of $d+1$ distinct sets in $\CB$ has dimension at most $-1$, and is thus empty. So it suffices to show the first part of the above statement.

We prove this by induction in the variable $k=|M|$. The assertion trivially holds for $k=1$. Now assume that the assertion holds for some natural number $k$. We show that it also holds for $k+1$.

Let $M=\set{\gamma(0),\dots,\gamma(k)}$ be a set of cardinality $k+1$ in $G$. Then obviously $X(M)\subset X(M')$ for every subset $M'\subset M$. Using the induction hypothesis, there exists a relatively compact, open set $A_0$ with $K\subset A_0\subset \quer{A}_0\subset V$, such that the collection $\set{\alpha_{\gamma(0)}(\del A_0),\dots,\alpha_{\gamma(k-1)}(\del A_0)}$ is in general position in $X$.

Since $\quer{A}_0\subset V\subset X(M)$, we can find for every point $x\in\del A_0$ a number $\eta(x)>0$ such that $\quer{B}_{\eta(x)}(x)\subset V$ and such that the sets $\alpha_{\gamma(j)}(B_{\eta(x)}(x))$ are pairwise disjoint for $j=0,\dots,k$. Denote $\widehat{B}_x = B_{\eta(x)}(x)$ and $B_x=B_{\eta(x)/2}(x)$.
Note that since $A_0$ was relatively compact, its boundary $\del A_0$ is compact. So find some finite subcover $\del A_0 \subset \bigcup_{i=1}^N B_i$. We will now construct relatively compact, open sets $A_i$ for $i=0,\dots,N$ ($A_0$ is already defined) with the following properties:
\begin{itemize}
\item[(1)] $\quer{A}_i \subset A_0\cup\bigcup_{j=1}^N B_j$.
\item[(2)] $A_{i}\subset A_{i+1}\subset A_{i}\cup \widehat{B}_{i+1}$.
\item[(3)] The collection 
\[\CA_i = \set{ \alpha_{\gamma(j)}(\del A_i)}_{j<k}\cup
\Bigl\{\alpha_{\gamma(k)}(\del A_i \cap \bigcup_{j=1}^i B_j) \Bigl\}\]
 is in general position. 
\end{itemize}
Once we have done this construction, combining (1) with (3) implies that the set $U=A_N$ has the desired property. It remains to show how to construct the sets $A_i$. 

So suppose that the set $A_i$ has already been defined for $i<N$. According to \eqref{D4}, for all nonempty subsets $S\subset\CA_i$, there exists a zero dimensional $F_\sigma$-set 
\[
E_S\subset\bigcap S \quad\text{with}\quad\dim(\bigcap S\setminus E_S) = \dim(\bigcap S)-1.
\]
Define
\begin{equation} \label{eq3:e1}
E:= \bigcup_{\emptyset\neq S\subset \CA_i \atop 0\leq j\leq k} \alpha_{\gamma(j)^{-1}}(E_S).
\end{equation}
By \eqref{D5}, $E$ is a zero dimensional $F_\sigma$-set. Use Lemma \ref{Startlemma} to find a relatively compact, open set $W$ such that
\begin{equation} \label{eq3:e2}
\quer{A_i\cap B_{i+1}} \subset W\subset\quer{W}\subset \widehat{B}_{i+1}\cap (A_0\cup\bigcup_{j=1}^N B_j)
\end{equation}
and
\begin{equation} \label{eq3:e3}
\del W \cap E=\emptyset.
\end{equation}
Now set $A_{i+1}:= A_i\cup W$. This clearly satisfies the properties (1) and (2). To show (3), let $\emptyset\neq S=\set{S_1,\dots,S_m}\subset \CA_{i+1}$ correspond to some subset $\set{j_1,j_2,\dots,j_m}\subset\set{0,\dots,k}$. Note that since $\del A_{i+1}\subset\del A_i\cup\del W$, we have either
\[S_l=\alpha_{\gamma(j_l)}(\del A_{i+1}) \subset \alpha_{\gamma(j_l)}(\del A_i)\cup 
\alpha_{\gamma(j_l)}(\del W) =: S_l^0\cup S_l^1\qquad(\text{if}\; j_l\neq k)\]
or
\[\begin{array}{ccc} 
S_l &=&\displaystyle \alpha_{\gamma(j_l)}(\del A_{i+1}\cap\bigcup_{j=1}^{i+1} B_i)\\\\
&\subset& \displaystyle \alpha_{\gamma(j_l)}((\del A_i\setminus W)\cap\bigcup_{j=1}^{i+1} B_i)\cup \alpha_{\gamma(j_l)}(\del W) \\\\
&\stackrel{\ref{eq3:e2}}{\subset}& \displaystyle \alpha_{\gamma(j_l)}(\del A_i\cap\bigcup_{j=1}^i B_j)\cup\alpha_{\gamma(j_l)}(\del W) \\\\
 &=:& S_l^0\cup S_l^1\qquad\qquad(\text{if}\; j_l=k).
\end{array}\]
It follows that
\[
\bigcap S \subset \bigcup_{a\in\set{0,1}^m} \left( \bigcap_{l=1}^m S_l^{a_l} \right).
\]
Since $\quer{W}\subset \widehat{B}_{i+1}$, our choice of $\widehat{B}_{i+1}$ implies that the sets $S_l^1$ are pairwise disjoint. So it suffices to consider the case $a=(0,\dots,0)$ and, since we can change the order without loss of generality, the case $a=(1,0,\dots,0)$. For $a=(0,\dots,0)$, note that $\set{S_1^0,\dots,S_m^0}$ is a subset of $\CA_i$, so we already have
\[\dim\Bigl(\bigcap_{l=1}^m S_l^0\Bigl) \quad\leq\quad \max(-1,n-m).\]
For $a=(1,0,\dots,0)$, define $\hat{S}=\set{S_2^0,\dots,S_m^0}$. This is a subset of $\CA_i$, hence we know that it is in general position. Moreover, considering our choice of the set $E_{\hat{S}}$, recall that
\[\dim(\bigcap \hat{S}\setminus E_{\hat{S}}) \leq\dim(\bigcap\hat{S})-1 \leq \max(-1,n-(m-1))-1 \leq \max(-1,n-m).\]
By the choice of $W$ we know that $\del W\cap E=\emptyset$, see \ref{eq3:e3}. Since $\alpha_{\gamma(j_1)^{-1}}(E_{\hat{S}})\subset E$ (see \ref{eq3:e1}), this implies $E_{\hat{S}}\cap\alpha_{\gamma(j_1)}(\del W)=\emptyset$. In particular, it follows that
\[S_1^1\cap\bigcap_{l=2}^m S_l^0 = \alpha_{\gamma(j_1)}(\del W)\cap\bigcap \hat{S} \subset \bigcap\hat{S}\setminus E_{\hat{S}}.\]
Therefore we have established
\[\dim(S_1^1\cap\bigcap_{l=2}^m S_l^0)\leq\max(-1,n-m).\]
If we combine these inequalities with \eqref{D2}, it follows that we have $\dim(\bigcap S) \leq \max(-1,n-m)$ as well. So $\CA_{i+1}$ is in general position and we are done.
\end{proof} 


Having established localized small boundary conditions out of local freeness properties of a group action, we can now also prove localized marker-type properties:

\begin{lemma}[cf.~{\cite[6.2]{gutman}} and {\cite[4.3]{szabo}}] \label{key lemma} 
Let $G$ be a group and $d\in\IN$ a natural number. Let $F\fin G$ be a finite subset and let $g_1,\dots,g_d\in G$ be group elements with the property that the sets
\[
 F^{-1}F ~,~ g_1 F^{-1}F ~,~\dots ~,~ g_d F^{-1}F
\]
are pairwise disjoint. Using the notation $g_0=1_G$, set $M=\bigcup_{l=0}^d g_l F^{-1}F$.

Let $X$ be a locally compact metric space and $\alpha: G\curvearrowright X$ an action. Then the following holds:

Let $U,V\subset X$ be relatively compact, open sets such that
\begin{itemize}
\item $\del U$ is $(M,d)$-disjoint;
\item $\quer{U}$ is $(F,1)$-disjoint;
\item $\quer{V}$ is $(M^{-1},1)$-disjoint.
\end{itemize}
Then there exists a relatively compact, open set $W\subset X$ such that $U\subset W,~ V\subset\bigcup_{g\in M} \alpha_g(W)$ and $\quer{W}$ is $(F,1)$-disjoint.
\end{lemma}
\begin{proof}

Set $R=\quer{V}\setminus\bigcup_{g\in M} \alpha_g(U)$. Observe that $R$ is compact and $(M^{-1},1)$-disjoint, so apply Lemma \ref{disjointness neighbourhood} and choose $\rho>0$ such that $\quer{B}_\rho(R)$ is compact and $(M^{-1},1)$-disjoint as well.
 We now claim that there exists a $\delta>0$ such that
\begin{equation} \label{eq3:e4}
|\set{g\in M|\; \alpha_g(\quer{U})\cap \quer{B}_\delta(x)\neq\emptyset}| \leq d\quad\text{for all}\; x\in R.
\end{equation}
Assume that this is not true. Let $x_n\in R$ be elements with $\delta_n>0$ such that $\delta_n\to 0$ and
\[
|\set{g\in M|\; \alpha_g(\quer{U})\cap \quer{B}_{\delta_n}(x_n)\neq\emptyset}|\geq d+1\quad\text{for all}\; n.
\]
By compactness, we can assume that $x_n$ converges to some $x\in R$ by passing to a subsequence. Moreover, since $M$ has only finitely many subsets, we can also assume (again by passing to a subsequence if necessary) that there are distinct $\gamma(0),\dots,\gamma(d)\in M$ such that $\alpha_{\gamma(l)}(\quer{U})\cap \quer{B}_{\delta_n}(x_n)\neq\emptyset$ for all $n$ and all $l=0,\dots,d$. But then $\delta_n\to 0$ implies
\[
x\in R\cap \bigcap_{l=0}^d \alpha_{\gamma(l)}(\quer{U}) \subset \bigcap_{l=0}^d \alpha_{\gamma(l)}(\del U) = \emptyset.
\]
So this gives a contradiction to $\del U$ being $(M,d)$-disjoint.
So we may choose a number $\delta\leq\rho$ satisfying \ref{eq3:e4}. Moreover, choose some finite covering
\[
R\subset\bigcup_{i=1}^s B_\delta(z_i)\quad\text{for some}\; z_1,\dots,z_s\in R.
\]
Note that the right-hand side is relatively compact and $(M^{-1},1)$-disjoint by our choice of $\rho$. Since the sets $\set{g_l F^{-1}F \;|\; l=0,\dots,d}$ are pairwise disjoint, observe that \ref{eq3:e4} enables us to define a map $c: \set{1,\dots,s}\to\set{0,\dots,d}$ such that
\begin{equation} \label{eq3:e5}
\alpha_g(\quer{U})\cap \quer{B}_\delta(z_i) = \emptyset\quad\text{for all}\; g\in g_{c(i)} F^{-1}F.
\end{equation}
Finally, set
\[
W=U\cup\bigcup_{i=1}^s \alpha_{g_{c(i)}^{-1}}(B_\delta(z_i)).
\]
Obviously, $W$ is a relatively compact, open set with $U\subset W$. Moreover, we have
\[
\begin{array}{ccc}
V &\subset& \displaystyle \bigcup_{g\in M} \alpha_g(U) \cup R \\
&\subset& \displaystyle  \bigcup_{g\in M} \alpha_g(U) \cup \bigcup_{i=1}^s~ \underbrace{B_\delta(z_i)}_{=\alpha_{g_{c(i)}}(\alpha_{g_{c(i)}^{-1}}(B_\delta(z_i)))} \\
&\subset& \displaystyle  \bigcup_{g\in M} \alpha_g(U) \cup \bigcup_{i=1}^s \alpha_{g_{c(i)}}(W) \quad \subset \bigcup_{g\in M} \alpha_g(W)
\end{array}
\]
At last we have to show that $\quer{W}$ is $(F,1)$-disjoint. Suppose that $\alpha_a(\quer{W})\cap\alpha_b(\quer{W})\neq\emptyset$ for some $a\neq b$ in $F$. That is, there exist $x,y\in\quer{W}$ such that $\alpha_a(x)=\alpha_b(y)$. Let us go through all the possible cases:
\begin{itemize}
\item $x,y\in \quer{U}$ is obviously impossible.
\item $x\in\alpha_{g_{c(i_1)}^{-1}}(\quer{B}_\delta(z_{i_1}))$ and $y\in\alpha_{g_{c(i_2)}^{-1}}(\quer{B}_\delta(z_{i_2}))$ for some $1\leq i_1, i_2\leq s$. It follows that
\[
\alpha_a(x)=\alpha_b(y)\in \alpha_{ag_{c(i_1)}^{-1}}(\quer{B}_\delta(z_{i_1}))\cap\alpha_{bg_{c(i_2)}^{-1}}(\quer{B}_\delta(z_{i_2})),
\]
so
\[
\begin{array}{ccl}
\emptyset &\neq& \alpha_{b^{-1} a g_{c(i_1)}^{-1}}(\quer{B}_\delta(z_{i_1}))\cap \alpha_{g_{c(2)}^{-1}}(\quer{B}_\delta(z_{i_2}))\\\\
&\subset& \alpha_{b^{-1}a g_{c(i_1)}^{-1}}(\quer{B}_\rho(R)))\cap \alpha_{g_{c(i_2)}^{-1}}(\quer{B}_\rho(R)).
\end{array}
\]
Observe that by $a\neq b$, we have $b^{-1}ag_{c(i_1)}^{-1}\neq g_{c(i_2)}^{-1}$ in $M^{-1}$.
Since $\quer{B}_\rho(R)$ is $(M^{-1},1)$-disjoint, the right side of the above is empty. So this is impossible.
\item $x\in\quer{U}$ and $y\in\alpha_{g_{c(i)}^{-1}}(\quer{B}_\delta(z_i))$ for some $1\leq i\leq s$. Then it follows that
\[
\alpha_a(x)=\alpha_b(y)\in\alpha_{a}(\quer{U})\cap\alpha_{bg_{c(i)}^{-1}}(\quer{B}_\delta(z_i))\neq\emptyset.
\]
Or equivalently, $\alpha_{g_{c(i)}b^{-1}a}(\quer{U})\cap \quer{B}_\delta(z_i)\neq\emptyset$, a contradiction to the definition of $c(i)$, see \ref{eq3:e5}.
\end{itemize}
So we see that $\quer{W}$ is indeed $(F,1)$-disjoint.
\end{proof}

The following result constitutes the main technical result of this Appendix:

\begin{lemma}[cf.~{\cite[6.1]{gutman}} and {\cite[4.4]{szabo}}] \label{local marker lemma} 
Let $G$ be a group and $d\in\IN$ a natural number. Let $F\fin G$ be a finite subset and let $g_0,g_1,\dots,g_d\in G$ be group elements with the property that the sets
\[
 g_0 F^{-1}F ~,~ g_1 F^{-1}F ~,~\dots ~,~ g_d F^{-1}F
\]
are pairwise disjoint. Set $M=\bigcup_{l=0}^d g_l F^{-1}F$.

Let $X$ be a locally compact metric space with an action $\alpha: G\curvearrowright X$ such that $(X,\alpha,G)$ has the $(M,d)$-small boundary property. Moreover, assume that $X(M^{-1})=X$. Then given any compact subset $K\subset X$, there exists a relatively compact, open set $Z\subset X$ with $K\subset \bigcup_{g\in M} \alpha_g(Z)$ and $\alpha_g(\quer{Z})\cap\alpha_h(\quer{Z})=\empty$ for all $g\neq h$ in $F$.
\end{lemma}
\begin{proof}
First we observe that shifting $M$ by a left translation does not affect the $(M,d)$-small boundary property or the property that $X(M^{-1}) = X$. Thus since $K$ was chosen arbitrarily, we may apply a left translation to $M$ and assume without loss of generality that $g_0 = 1_G$.

For all $x\in K$, use $X(M^{-1})=X$ to choose a relatively compact, open neighbourhood $U_x$ such that $\quer{U}_x$ is $(M^{-1},1)$-disjoint. By the $(M,d)$-small boundary property, we can also assume that $\del U_x$ is $(M,d)$-disjoint. Note that since $1_G\in F^{-1}F\subset M$, it follows that every $(M^{-1},1)$-disjoint set is also $(F,1)$-disjoint.

Choose a finite subcovering $K\subset\bigcup_{i=0}^s U_i$. Apply Lemma \ref{key lemma} (with respect to $U=U_0, V=U_1$) to find a relatively compact, open set $W_1$ such that $U_0\subset W_1, U_1\subset\bigcup_{g\in M} \alpha_g(W)$ and such that $\quer{W}_1$ is $(F,1)$-disjoint. Clearly we have $U_0\cup U_1\subset\bigcup_{g\in M} \alpha_g(W_1)$.

Now carry on inductively. If $W_k$ is already defined, apply Lemma \ref{key lemma} (with respect to $U=W_k, V=U_{k+1}$) to find a relatively compact, open set $W_{k+1}$ such that $W_k\subset W_{k+1}$ and $U_{k+1}\subset \bigcup_{g\in M} \alpha_g(W_{k+1})$ and such that $\quer{W}_{k+1}$ is $(F,1)$-disjoint. Note also that if $W_k$ had the property that 
\[
U_0\cup\dots\cup U_k\subset\bigcup_{g\in M}\alpha_g(W_k),\]
then it follows that
\[\begin{array}{ccl} U_0\cup\dots\cup U_k\cup U_{k+1} &\subset& \bigcup_{g\in M} \alpha_g(W_k) \cup U_{k+1} \\\\
&\subset& \bigcup_{g\in M} \alpha_g(W_k) \cup \bigcup_{g\in M} \alpha_g(W_{k+1}) \\\\
&=& \bigcup_{g\in M} \alpha_g(W_{k+1}).
\end{array}
\]
So set $Z=W_s$. The set $\quer{Z}$ is compact and $(F,1)$-disjoint by construction and moreover satisfies  
\[ 
K\subset U_0\cup\dots\cup U_s\subset\bigcup_{g\in M} \alpha_g(Z).
\]
\end{proof}

\bibliographystyle{alpha}
\bibliography{nuclear-dimension-abelian}
\end{document}